\documentclass[a4paper,12pt,centertags,oneside]{amsart}
\usepackage{amsmath,amstext,amsthm,a4,amssymb,amscd}
\usepackage[mathscr]{eucal}
\usepackage{mathrsfs}
\usepackage{bbold}
\usepackage{epsf}
\usepackage{typearea}
\usepackage{charter}

\usepackage{graphicx}
\usepackage[all]{xy}

\newcommand{\R}{\mathbb{R}}
\newcommand{\C}{\mathbb{C}}
\newcommand{\Z}{\mathbb{Z}}
\newcommand{\N}{\mathbb{N}}

\newcommand{\smooth}{{\mathscr{C}^\infty}}

\makeatletter
\newcommand*{\rom}[1]{\expandafter\@slowromancap\romannumeral #1@}
\makeatother

\newtheorem{prop}{Proposition}[section]
\newtheorem{thm}[prop]{Theorem}

\newtheorem{cor}[prop]{Corollary}

\theoremstyle{definition}
\newtheorem{defn}[prop]{Definition}

\theoremstyle{remark}
\newtheorem{rem}[prop]{Remark}

\numberwithin{equation}{section}
\title{An extension of BCOV invariant}
\date{\today}

\author{Yeping ZHANG}
\address{School of Mathematics,
Korea Institute for Advanced Study,
Hoegiro 85, Dongdaemungu,
Seoul 02455, Korea}
\email{ypzhang@kias.re.kr}

\begin{document}

\begin{abstract}
Bershadsky, Cecotti, Ooguri and Vafa
constructed a real valued invariant for Calabi-Yau manifolds,
which is called the BCOV invariant.
In this paper,
we consider a pair $(X,Y)$,
where $X$ is a compact K{\"a}hler manifold
and $Y\in\big|K_X^m\big|$ with $m\in\Z\backslash\{0,-1\}$.
We extend the BCOV invariant to such pairs.
If $m=-2$ and $X$ is a rigid del Pezzo surface,
the extended BCOV invariant is equivalent to Yoshikawa's equivariant BCOV invariant.
If $m=1$,
the extended BCOV invariant is well-behaved under blow-up.
It was conjectured that
birational Calabi-Yau threefolds have the same BCOV invariant.
As an application of our extended BCOV invariant,
we show that this conjecture holds for Atiyah flops. \\
Keywords: analytic torsion, Calabi-Yau manifolds, birational maps. \\
MSC classification: 58J52.
\end{abstract}

\maketitle

\tableofcontents

\section{Introduction}
\label{sect-intro}

The BCOV torsion is
a real valued invariant for Calabi-Yau manifolds equipped with Calabi-Yau metrics.
Bershadsky, Cecotti, Ooguri and Vafa initiated
the study of BCOV torsion for Calabi-Yau threefolds
in the outstanding papers \cite{bcov,bcov2}.
Their work extended the mirror symmetry conjecture of
Candelas, de la Ossa, Green and Parkes \cite{cdgp}.
Fang and Lu \cite{fl}
studied the BCOV torsion for Calabi-Yau manifolds of arbitrary dimension.

The BCOV invariant is
a real valued invariant for Calabi-Yau manifolds,
which could be viewed as a normalization of the BCOV torsion.
Fang, Lu and Yoshikawa \cite{fly} constructed and studied
the BCOV invariant for Calabi-Yau threefolds.
Their work confirmed a conjecture of
Bershadsky, Cecotti, Ooguri and Vafa \cite{bcov,bcov2}
concerning the BCOV torsion of quintic mirror threefolds.
Eriksson, Freixas and Mourougane \cite{efm,efm2}
extended these results to Calabi-Yau manifolds of arbitrary dimension.

For a Calabi-Yau manifold $X$,
we denote by $\tau(X)$
the logarithm of the BCOV invariant of $X$ defined in \cite{efm}.

Fang, Lu and Yoshikawa \cite[Conjecture 4.17]{fly} conjectured the following:
for a pair of birationally equivalent Calabi-Yau threefolds $(X,X')$,
we have
\begin{equation}
\tau(X') - \tau(X) = \nu(X,X') \;,
\end{equation}
where $\nu(X,X')$ is a real number determined by the topology of $(X,X')$.
In particular,
for two families of Calabi-Yau threefolds $\big(X_s\big)_{s\in S}$ and $\big(X_s'\big)_{s\in S}$
such that $X_s$ and $X_s'$ are birational for any $s\in S$,
the conjecture implies that the function
\begin{align}
\begin{split}
S & \rightarrow \R \\
s & \mapsto \tau(X_s') - \tau(X_s)
\end{split}
\end{align}
is locally constant.
In an earlier paper \cite[Conjecture 2.1]{yo06},
Yoshikawa made a stronger conjecture:
for a pair of birationally equivalent Calabi-Yau threefolds $(X,X')$,
we have
\begin{equation}
\tau(X') = \tau(X) \;.
\end{equation}

Let $X$ and $X'$ be projective Calabi-Yau threefolds defined over a field $L$.
Let $T$ be a finite set of embeddings $L\hookrightarrow\C$.
For $\sigma\in T$,
we denote by $X_\sigma$ (resp. $X_\sigma'$) the base change of $X$ (resp. $X'$) to $\C$ via the embedding $\sigma$.
We denote by $D(X_\sigma)$ (resp. $D(X_\sigma')$) the derived category of coherent sheaves on $X_\sigma$ (resp. $X_\sigma'$).
Maillot and R{\"o}ssler \cite[Theorem 1.1]{ma-ro} showed that
if one of the following conditions holds,
\begin{itemize}
\item[(a)] there exists $\sigma\in T$ such that $X_\sigma$ and $X_\sigma'$ are birational,
\item[(b)] there exists $\sigma\in T$ such that $D(X_\sigma)$ and $D(X_\sigma')$ are equivalent as triangulated $\C$-linear categories,
\end{itemize}
then there exist a positive integer $n$ and a non-zero element $\alpha\in L$
such that for any $\sigma\in T$,
we have
\begin{equation}
\tau(X_\sigma') - \tau(X_\sigma) = \frac{1}{n} \log \big|\sigma(\alpha)\big| \;.
\end{equation}
While a result of Bridgeland \cite[Theorem 1.1]{brid} showed that (a) implies (b),
Maillot and R{\"o}ssler gave two separated proofs for (a) and (b).

The purpose of this paper is to extend the BCOV invariant to certain non Calabi-Yau cases
and establish several properties of this extension.
One of our results says that
Yoshikawa's conjecture \cite[Conjecture 2.1]{yo06} holds for Atiyah flops.

Since the BCOV invariant is defined as a product of certain Quillen metrics \cite{q},
the work of Bismut, Gillet and Soul{\'e} \cite{bgs1,bgs2,bgs3} on the Quillen metric
is of fundamental importance in this paper.
The immersion formula of Bismut and Lebeau \cite{ble} is
another powerful tool which we will use.

Let us now give more detail about the matter of this paper.

Let $X$ be a compact K{\"a}hler manifold.
Let $K_X$ be its canonical bundle.
Let $m\in\Z\backslash\{0,-1\}$.
Let $K_X^m$ be the $m$-th tensor power of $K_X$.
We assume that $H^0(X,K_X^m)\neq 0$.
Let $\gamma \in H^0(X,K_X^m)\backslash\{0\}$.
Let $Y$ be the zero locus of $\gamma$.
We assume that $Y$ is smooth and reduced.
We call $(X,Y)$ an $m$-Calabi-Yau pair.
We will construct a real number $\tau(X,Y)$
determined by the complex structure of $(X,Y)$.
If $X$ is Calabi-Yau,
then
\begin{equation}
\tau(X,\emptyset) = \tau(X) \;.
\end{equation}

\hfill \\
\noindent\textbf{Curvature of $\tau(X,Y)$.}
Let $\pi_\mathscr{X}: \mathscr{X}\rightarrow S$ be a locally K{\"a}hler holomorphic fibration,
i.e., for any $s\in S$, there exists an open subset $s\in U\subseteq S$ such that $\pi^{-1}(U)$ is K{\"a}hler.
For $s\in S$,
we denote $X_s = \pi_\mathscr{X}^{-1}(s)$.
Let $\mathscr{Y}\subseteq\mathscr{X}$ be a complex hypersurface
such that the restricted map
$\pi_\mathscr{Y} := \pi_\mathscr{X}\big|_\mathscr{Y}$ is a fibration.
For $s\in S$,
we denote $Y_s = \pi_\mathscr{Y}^{-1}(s)$.
Let $m\in\Z\backslash\{0,-1\}$.
We assume that
$(X_s,Y_s\big)$ is an $m$-Calabi-Yau pair for any $s\in S$.
Let $\tau(X,Y)$ be
the function $s\mapsto\tau(X_s,Y_s)$ on $S$.
One of our central results is a formula
relating $\overline{\partial}\partial\tau(X,Y)$
to the Hodge form and the Weil-Petersson form.

First we introduce the Hodge form.
Let $n=\dim X$.
Let $\big(g^{H^{p,q}(X)}\big)_{0\leqslant p,q\leqslant n}$
be Hermitian metrics on the holomorphic vector bundles
$\big(H^{p,q}(X)\big)_{0\leqslant p,q\leqslant n}$ over $S$
such that $g^{H^{p,q}(X)}(u,u)=g^{H^{q,p}(X)}(\overline{u},\overline{u})$
for any $p,q$ and any $u\in H^{p,q}(X)$.
The Hodge form associated with the variation of Hodge structure $H^\bullet(X)$ can be defined by
\begin{equation}
\label{eq-intro-hform}
\omega_{H^\bullet(X)} = \frac{1}{2} \sum_{0\leqslant p,q\leqslant n}
(-1)^{p+q}(p-q)c_1\big(H^{p,q}(X),g^{H^{p,q}(X)}\big)
\in \Omega^{1,1}(S) \;.
\end{equation}
We can show that
$\omega_{H^\bullet(X)}$ is independent of $\big(g^{H^{p,q}(X)}\big)_{0\leqslant p,q\leqslant n}$.
If $X_s$ is Calabi-Yau,
then \eqref{eq-intro-hform} is
the Hermitian form of the Hodge metric considered in \cite{fl}.
The Hodge form $\omega_{H^\bullet(Y)}\in \Omega^{1,1}(S)$
associated with $H^\bullet(Y)$
can be defined in the same way.

Now we introduce the Weil-Petersson form.
Locally we have a holomorphic map $s\mapsto\gamma_s \in H^0(X_s,K_{X_s}^m)$
such that $Y_s$ is the zero locus of $\gamma_s$.
The function
\begin{equation}
\label{eq0-intro-wp}
P: \; s \mapsto \log \int_{X_s}\big|\gamma_s\overline{\gamma_s}\big|^{1/m}
\end{equation}
is well-defined up to pluriharmonic functions,
where $\big|\gamma_s\overline{\gamma_s}\big|^{1/m}$
is the unique real semi-positive $(n,n)$-form on $X_s$
whose $m$-th tensor power equals $e^{i\theta}\gamma_s\wedge\overline{\gamma_s}$ for certain $\theta\in\R$.
The Weil-Petersson form is defined by
\begin{equation}
\label{eq-intro-wp}
\omega_{\pi_\mathscr{X},\pi_\mathscr{Y}}
= - \frac{\overline{\partial}\partial P}{2\pi i} \in \Omega^{1,1}(S) \;.
\end{equation}
If $X_s$ is Calabi-Yau,
then \eqref{eq-intro-wp} is the K{\"a}hler form
of the Weil-Petersson metric considered in \cite{chs,t} (cf. \cite{j}).

Let $\chi(X)$ (resp. $\chi(Y)$)
be the Euler number of $X$ (resp. $Y$).
We denote
\begin{equation}
\label{eq-def-w}
w_m(X) = \frac{\chi(X)}{12}
- \frac{\chi(Y)}{12(m+1)} \;.
\end{equation}

\begin{thm}
\label{intro-thm-curvature}
The following identity holds,
\begin{equation}
\label{eq-intro-thm-curvature}
\frac{\overline{\partial}\partial}{2\pi i} \tau(X,Y) =
\omega_{H^\bullet(X)} - \frac{1}{m+1}\omega_{H^\bullet(Y)}
- w_m(X) \omega_{\pi_\mathscr{X},\pi_\mathscr{Y}} \;.
\end{equation}
\end{thm}

The proof of Theorem \ref{intro-thm-curvature}
is based on \cite[Theorem 0.1]{bgs1}.

\hfill \\
\noindent\textbf{Relation with Yoshikawa's equivariant BCOV invariant.}
Let $(X,Y)$ be a $(-2)$-Calabi-Yau pair such that
$X$ is a del Pezzo surface and $K_X^{-2}$ is very ample.
Let $(\cdot,\cdot)$ be the intersection form on $\mathrm{Pic}(X)$.
By \cite[Chapter III, 3.4 Proposition]{kol},
if $(K_X,K_X)\geqslant 2$,
then $K_X^{-2}$ is very ample.

Let $X'\rightarrow X$ be the ramified double covering whose branch locus is $Y$.
Then $X'$ is a $2$-elementary K3 surface, i.e.,
$X'$ is a K3 surface equipped with an involution $\iota$ commuting with $X'\rightarrow X$.
Yoshikawa \cite{y04} constructed an equivariant BCOV invariant
for $2$-elementary K3 surfaces
(denoted $\tau_M$ in Yoshikawa's paper).
We denote by $\tau(X',\iota)$
the logarithm of Yoshikawa's equivariant BCOV invariant of $(X',\iota)$.

\begin{thm}
\label{intro-thm-ex}
Let $(X,Y)$ and  $(X',\iota)$ be as above.
We have
\begin{equation}
\tau(X,Y) = - \tau(X',\iota) + \nu(X) \;,
\end{equation}
where $\nu(X)$ is a real number determined by $X$.
\end{thm}

The proof of Theorem \ref{intro-thm-ex}
is based on Theorem \ref{intro-thm-curvature}
and a special case of \cite[Theorem 0.1]{ble}
(see Theorem \ref{thm-quillen-immersion}).

Let $g$ be the genus of the curve $Y$.
By \cite[Theorem 0.1]{ma-yo},
the function
\begin{equation}
(X',\iota) \mapsto \exp \Big(-2^{g}(2^g+1)\tau(X',\iota)\Big)
\end{equation}
on the moduli space of the $2$-elementary K3 surfaces in question
is the product of a Borcherds product \cite{bor} and a Siegel modular form.
By Theorem \ref{intro-thm-ex},
the same result holds for $\tau(X,Y)$
if $X$ admits no deformation,
which holds for $(K_X,K_X) > 5$.
In this case,
the number $\nu(X)$ is uniquely determined by the topological type of $X$.

\hfill \\
\noindent\textbf{Behavior of $\tau(X,Y)$ under blow-up.}
Let $(X,Y)$ be a $1$-Calabi-Yau pair.
Let $Z \subseteq X$ be a closed complex submanifold of codimension $2$
such that $Z \cap Y = \emptyset$.
Let $f: X' \rightarrow X$ be the blow-up along $Z$.
Set $Y'=f^{-1}(Y \cup Z)\subseteq X'$.
Then $(X',Y')$ is a $1$-Calabi-Yau pair.
We are interested in the value of
\begin{equation}
\label{eq-intro-blq}
\tau(X',Y') - \tau(X,Y) \;.
\end{equation}
Here the technical conditions $Z\cap Y = \emptyset$ and $\mathrm{codim} Z=2$
are due to our hypothesis that the canonical divisor $Y'$ is smooth and reduced.

\begin{thm}
\label{intro-thm-bl}
There exists $\nu\in\R$ such that
for any $X,Y,Z,X',Y'$ as above with $\dim X = 2$
and $Z$ being a single point,
we have
\begin{equation}
\tau(X',Y') - \tau(X,Y) = \nu \;.
\end{equation}
\end{thm}

A curve in a threefold is called a $(-1,-1)$-curve if
\begin{itemize}
\item[-] it is isomorphic to ${\C}P^1$;
\item[-] its normal bundle is isomorphic to the sum of two line bundles of degree $-1$.
\end{itemize}

\begin{thm}
\label{intro-thm-bl-curve}
There exists $\nu\in\R$ such that
for any $X,Y,Z,X',Y'$ as above with $\dim X = 3$
and $Z$ being a $(-1,-1)$-curve,
we have
\begin{equation}
\tau(X',Y') - \tau(X,Y) = \nu \;.
\end{equation}
\end{thm}

The key step in the proof of Theorem \ref{intro-thm-bl}, \ref{intro-thm-bl-curve}
is an application of \cite[Theorem 8.10]{b97}
(see Theorem \ref{thm-quillen-bl} and Corollary \ref{cor-quillen-bl}).

Let $X$ be a Calabi-Yau threefold.
We assume that there is a $(-1,-1)$-curve $Z \subseteq X$.
Let $f: X' \rightarrow X$ be the blow-up along $Z$.
Set $D = f^{-1}(Z) \subseteq X'$.
We have $D \simeq {\C}P^1 \times {\C}P^1$.
Let
\begin{equation}
\mathrm{pr}_1,\mathrm{pr}_2: {\C}P^1 \times {\C}P^1 \rightarrow {\C}P^1
\end{equation}
be the projections to the first and the second component.
We identify $D$ with ${\C}P^1 \times {\C}P^1$
such that $f\big|_D = \mathrm{pr}_1$.
There exists a blow-down
\begin{equation}
g: X' \rightarrow X''
\end{equation}
such that $g\big|_{X'\backslash D}$ is biholomorphic and $g\big|_D = \mathrm{pr}_2$.
The birational map
\begin{equation}
g \circ f^{-1}: X \dashrightarrow X''
\end{equation}
is called an Atiyah flop.
Here $X''$ is Calabi-Yau
and $g(D)\subseteq X''$ is a $(-1,-1)$-curve.

\begin{cor}
\label{intro-thm-flop}
Let $X \dashrightarrow X''$ be an Atiyah flop between
Calabi-Yau K{\"a}hler threefolds.
We have
\begin{equation}
\tau(X,\emptyset) = \tau(X'',\emptyset) \;.
\end{equation}
\end{cor}

The proof of Corollary \ref{intro-thm-flop} is immediate:
let $\nu\in\R$ be as in Theorem \ref{intro-thm-bl-curve},
then both $\tau(X,\emptyset)$ and $\tau(X'',\emptyset)$
are equal to $\tau(X',D) - \nu$.

By Corollary \ref{intro-thm-flop},
the conjecture \cite[Conjecture 2.1]{yo06} holds for Atiyah flops.

\begin{rem}
\label{rem-nami}
We remark that
an Atiyah flop of a projective Calabi-Yau threefold
may fail to be projective.
Readers may find examples and counterexamples from a paper of Namikawa \cite{nami}.
There Namikawa considered a projective Calabi-Yau threefold $X$,
which is the fiber product of two rational elliptic surfaces.
The Calabi-Yau threefold $X$ contains a large number of $(-1,-1)$-curves.
Taking the Atiyah flop of several of them,
we may get either a projective Calabi-Yau threefold or a non projective one.
\end{rem}

This paper is organized as follows.
In \textsection \ref{sect-pre},
we introduce several fundamental notions and constructions.
In \textsection \ref{sect-bcov},
we construct the BCOV invariant $\tau(X,Y)$
and establish Theorem \ref{intro-thm-curvature}.
In \textsection \ref{sect-ex},
we establish Theorem \ref{intro-thm-ex}.
In \textsection \ref{sect-bl},
we establish Theorem \ref{intro-thm-bl}, \ref{intro-thm-bl-curve}
together with a weak result about \eqref{eq-intro-blq} in arbitrary dimension.

\hfill\\
\noindent\textbf{Notations.}
For a complex vector space $V$,
we denote $\det V = \Lambda^{\dim V} V$,
which is a complex line.
For a complex line $\lambda$,
we denote by $\lambda^{-1}$ the dual of $\lambda$.
For a graded complex vector space $V^\bullet = \bigoplus_{k=0}^m V^k$,
we denote $\det V^\bullet = \bigotimes_{k=0}^m \big(\det V^k\big)^{(-1)^k}$.

For $p,q\in\N$ and a complex vector bundle $F$ over a complex manifold $S$,
we denote by $\Omega^{p,q}(S,F)$ (resp. $A^{p,q}(S,F)$)
the vector space of $(p,q)$-forms (resp. $(p,q)$-current)
on $S$ with values in $F$.
We denote $\Omega^{p,q}(S)=\Omega^{p,q}(S,\C)$
(resp. $A^{p,q}(S)=A^{p,q}(S,\C)$).
For a differential form (resp. current) $\omega$ on $S$,
its component of degree $(p,q)$ is denoted by
$\big\{\omega\big\}^{(p,q)}$.

For a holomorphic vector bundle $E$ over a complex manifold $S$,
we denote by $\mathscr{O}_S(E)$
the analytic coherent sheaf of holomorphic sections of $E$.
We denote $\mathscr{O}_S = \mathscr{O}_S(\C)$.
For $p\in\N$,
we denote by $\Omega^p_S$
the analytic coherent sheaf of holomorphic $p$-forms on $S$.
We denote $\Omega^p_S(E) = \Omega^p_S \otimes_{\mathscr{O}_S} \mathscr{O}_S(E)$.

For $q\in\N$ and an analytic coherent sheaf $\mathscr{F}$ on a complex manifold $S$,
we denote by $H^q(S,\mathscr{F})$
the $q$-th cohomology of $\mathscr{F}$.
For a holomorphic vector bundle $E$ over  $S$,
we denote $H^q(S,E) = H^q(S,\mathscr{O}_S(E))$.
If $H^0(S,E)\neq 0$,
we denote $\big|E\big| = \mathbb{P}\big(H^0(S,E)\big)$.
We denote by $\mathscr{M}(S,E)$ the vector space of meromorphic sections of $E$.

For $k\in \N$ and a complex manifold $S$,
we denote by $H^k_\mathrm{dR}(S)$ the $k$-th de Rham cohomology of $S$
with coefficients in $\C$.
For $p,q\in \N$,
we denote $H^{p,q}(S) = H^q(S,\Omega^p_S)$.
If $S$ is a compact K{\"a}hler manifold,
we identify $H^{p,q}(S)$ with a sub vector space of $H^{p+q}_\mathrm{dR}(S)$
via the Hodge theory.

\hfill\\
\noindent\textbf{Acknowledgments.}
The author is grateful to Professor Ken-Ichi Yoshikawa
who is the author's postdoctoral advisor
and a member of the author's dissertation committee.
Prof. Yoshikawa drew the author's attention to the BCOV invariant
and suggested the author to
study the special case $m=-2$.

The author is grateful to Professor Jean-Pierre Demailly
who taught the author to prove Proposition \ref{prop-a}.

The author is grateful to Professor Yoshinori Namikawa
who taught the author the example given in Remark \ref{rem-nami}.

The author is grateful to Professor Xianzhe Dai and Professor Vincent Maillot
for many helpful suggestions.

The author is grateful to Doctor Yang Cao
who was the author's neighbor.
Dr. Cao taught the author a lot about algebraic geometry
and related topics.

This paper combines the preprints \cite{z2,z1}.

This work was supported by JSPS KAKENHI Grant Number JP17F17804.

This work was supported by KIAS individual Grant MG077401 at Korea Institute for Advanced Study.

\section{Preliminary}
\label{sect-pre}

\subsection{Chern form and Bott-Chern form}
\label{subsect-bc}

Let $S$ be a complex manifold.
Let $E$ be a holomorphic vector bundle over $S$.
Let $g^E$ be a Hermitian metric on $E$.
Let
\begin{equation}
R^E \in \Omega^{1,1}(S,\mathrm{End}(E))
\end{equation}
be the Chern curvature of $(E,g^E)$.
For $k\in\N$,
we denote by $c_k$ the $k$-th elementary symmetric polynomial.
The $k$-th Chern form of $(E,g^E)$ is defined by
\begin{equation}
c_k(E,g^E) := c_k \Big(-\frac{R^E}{2\pi i}\Big) \in \Omega^{k,k}(S) \;.
\end{equation}
The $k$-th Chern class of $E$ is defined by
\begin{equation}
c_k(E) := \big[c_k(E,g^E)\big] \in H^{2k}_\mathrm{dR}(S) \;,
\end{equation}
which is independent of $g^E$.

We denote
\begin{equation}
c(E,g^E) = 1 + c_1(E,g^E) + c_2(E,g^E) +...
\in \bigoplus_{k\in\N} \Omega^{k,k}(S) \;.
\end{equation}
The total Chern class of $E$ is defined by
\begin{equation}
c(E):=\big[c(E,g^E)\big]\in H^\mathrm{even}_\mathrm{dR}(S) \;.
\end{equation}
For a short exact sequence of holomorphic vector bundles over $S$,
\begin{equation}
0 \rightarrow E' \rightarrow E \rightarrow E'' \rightarrow 0 \;,
\end{equation}
we have
\begin{equation}
c(E) = c(E')c(E'') \;.
\end{equation}
Let $g^E$ be a Hermitian metric on $E$.
Let $g^{E'}$ be the Hermitian metric on $E'$
induced by $g^E$ via the embedding $E'\rightarrow E$.
Let $g^{E''}$ be the quotient Hermitian metric on $E''$
induced by $g^E$ via the surjection $E\rightarrow E''$.
The Bott-Chern form \cite[Section 1f)]{bgs1}
\begin{equation}
\widetilde{c}(E',E,g^E) \in
\bigoplus_{k\in\N}
\frac{\Omega^{k,k}(S)}{\partial\Omega^{k-1,k}(S)+\overline{\partial}\Omega^{k,k-1}(S)}
\end{equation}
is such that
\begin{equation}
\label{eq-def-BC}
\frac{\overline{\partial}\partial}{2\pi i} \widetilde{c}(E',E,g^E)
= c(E,g^E) - c(E',g^{E'})c(E'',g^{E''}) \;.
\end{equation}

\subsection{Hodge form}
\label{subsect-hodge}

Let $S$ be a complex manifold.
Let $H^\bullet_\Z = \bigoplus_{k=0}^n H^k_\Z$
be a local system of finitely generated graded $\Z$-module over $S$.
We denote $H^\bullet_\C = H^\bullet_\Z \otimes_\Z \C$,
which is a graded flat complex vector bundle over $S$.

For $k=0,\cdots,n$,
let
\begin{equation}
\label{eq-filt-hoge}
H^k_\C = F^0 H^k_\C
\supseteq F^1 H^k_\C
\supseteq \cdots
\supseteq F^k H^k_\C
\supseteq F^{k+1}H^k_\C = 0
\end{equation}
be a filtration by holomorphic sub vector bundles.
We assume that there exists a decomposition
by smooth complex sub vector bundles
\begin{equation}
\label{eq-sum-hoge}
H^k_\C = \bigoplus_{0\leqslant p,q\leqslant k, p+q=k} H^{p,q}_\C
\end{equation}
such that
\begin{equation}
\label{eq2-sum-hoge}
F^r H^k_\C = \bigoplus_{p=r}^k H^{p,k-p}_\C \;,\hspace{5mm}
\overline{H^{p,q}_\C} = H^{q,p}_\C \;.
\end{equation}
We remark that
\begin{equation}
\label{eq3-sum-hoge}
H^{p,q}_\C = F^p H^{p+q}_\C \cap \overline{F^q H^{p+q}_\C} \;.
\end{equation}
As a consequence,
the decomposition \eqref{eq-sum-hoge}
is uniquely determined by the filtration \eqref{eq-filt-hoge}.
Moreover, the identification
\begin{equation}
\label{eq4-sum-hoge}
H^{p,q}_\C = F^pH^{p+q}_\C / F^{p+1}H^{p+q}_\C
\end{equation}
induces a holomorphic structure on $H^{p,q}_\C$.
We call $H^\bullet := (H^\bullet_\Z, F^\bullet H^\bullet_\C)$
a variation of Hodge structure over $S$.
This definition of variation of Hodge structure
is weaker than the usual one,
which requires Griffiths transversality (cf. \cite[Definition 7.3.4]{cat}).

Set
\begin{equation}
\label{eq-def-lambda-lambdadR}
\lambda = \bigotimes_{0\leqslant p,q\leqslant n}
\Big(\det H^{p,q}_\C\Big)^{(-1)^{p+q}p} \;,\hspace{5mm}
\lambda_\mathrm{dR} = \bigotimes_{k=1}^n
\Big(\det H^k_\C\Big)^{(-1)^kk} \;.
\end{equation}
Then $\lambda$ (resp. $\lambda_\mathrm{dR}$) is
a holomorphic (resp. flat) line bundle over $S$.
By \eqref{eq-sum-hoge},
the second identity in \eqref{eq2-sum-hoge}
and \eqref{eq-def-lambda-lambdadR},
we have
\begin{equation}
\label{eq-lambda-lambdadR}
\lambda_\mathrm{dR} = \lambda \otimes \overline{\lambda} \;.
\end{equation}
The identity \eqref{eq-lambda-lambdadR} appeared in a paper of Kato \cite[\textsection 1]{kato}.

Let $U\subseteq S$ be a small open subset.
Let $\tau\in H^0(U,\lambda)$ be a nowhere vanishing holomorphic section.
Let $\sigma\in \smooth(U,\lambda_\mathrm{dR})$ be a non-zero constant section.
By \eqref{eq-lambda-lambdadR},
there exists $f\in\smooth(U,\C)$ such that
\begin{equation}
\label{eq-def-f}
\sigma = e^f \tau \otimes \overline{\tau} \;.
\end{equation}
Then $\overline{\partial}\partial\,\mathrm{Re} f \in \Omega^{1,1}(U)$
is independent of $\tau$ and $\sigma$.
The Hodge form $\omega_{H^\bullet}\in\Omega^{1,1}(S)$ is defined by
\begin{equation}
\label{eq-def-hodgeform}
\omega_{H^\bullet}\big|_U = \frac{\overline{\partial}\partial\,\mathrm{Re} f}{2\pi i} \;.
\end{equation}

Let $g^{H^\bullet_\C}$ be a Hermitian metric on $H^\bullet_\C$ such that
\begin{align}
\label{eq-metric-vhs}
\begin{split}
g^{H^\bullet_\C}(u,v) = 0  & \;,
\hspace{5mm} \text{for } u\in H^{p,q}_\C, v\in H^{p',q'}_\C
\text{ with } (p,q)\neq(p',q') \;, \\
g^{H^\bullet_\C}(u,u) = g^{H^\bullet_\C}(\overline{u},\overline{u}) & \;,
\hspace{5mm} \text{for } u\in H^\bullet_\C \;.
\end{split}
\end{align}
Let $g^{H^{p,q}_\C}$ be the restriction of $g^{H^\bullet_\C}$ to $H^{p,q}_\C$.
Let $c_1\big(H^{p,q}_\C,g^{H^{p,q}_\C}\big)\in\Omega^{1,1}(S)$
be its first Chern form.
Here $H^{p,q}_\C$ is viewed a holomorphic vector bundle in the sense of \eqref{eq4-sum-hoge}.

\begin{prop}
\label{prop-hodgeform}
The following identity holds,
\begin{equation}
\label{eq-prop-hodgeform}
\omega_{H^\bullet} = \frac{1}{2} \sum_{0\leqslant p,q\leqslant n}
(-1)^{p+q}(p-q)c_1\big(H^{p,q}_\C,g^{H^{p,q}_\C}\big) \;.
\end{equation}
\end{prop}
\begin{proof}
Let \hbox{$\big\lVert\cdot\big\rVert_\lambda$}
(resp. \hbox{$\big\lVert\cdot\big\rVert_{\lambda_\mathrm{dR}}$})
be the norm on $\lambda$ (resp. $\lambda_\mathrm{dR}$)
induced by $g^{H^\bullet_\C}$.

Let $U\subseteq S$ be a small open subset.
Let $\tau$, $\sigma$ and $f$ be as in \eqref{eq-def-f}.
By \eqref{eq-def-f} and \eqref{eq-metric-vhs},
we have
\begin{equation}
\label{eq1-pf-prop-hodgeform}
\mathrm{Re} f =
- \log \big\lVert\tau \big\rVert^2_{\lambda}
+ \frac{1}{2} \log \big\lVert\sigma\big\rVert^2_{\lambda_\mathrm{dR}} \;.
\end{equation}
By the Poincar{\'e}-Lelong formula,
\eqref{eq-def-hodgeform} and \eqref{eq1-pf-prop-hodgeform},
we have
\begin{equation}
\label{eq2-pf-prop-hodgeform}
\omega_{H^\bullet} =
c_1\big(\lambda,\big\lVert\cdot\big\rVert_\lambda\big)
-\frac{1}{2} c_1\big(\lambda_\mathrm{dR},\big\lVert\cdot\big\rVert_{\lambda_\mathrm{dR}}\big) \;.
\end{equation}
On the other hand,
by \eqref{eq-sum-hoge}, \eqref{eq2-sum-hoge}, \eqref{eq-def-lambda-lambdadR} and \eqref{eq-metric-vhs},
we have
\begin{align}
\label{eq3-pf-prop-hodgeform}
\begin{split}
c_1\big(\lambda,\big\lVert\cdot\big\rVert_\lambda\big)
& = \sum_{0\leqslant p,q\leqslant n}
(-1)^{p+q}pc_1\big(H^{p,q}_\C,g^{H^{p,q}_\C}\big) \;,\\
c_1\big(\lambda_\mathrm{dR},\big\lVert\cdot\big\rVert_{\lambda_\mathrm{dR}}\big)
& = \sum_{0\leqslant p,q\leqslant n}
(-1)^{p+q}(p+q)c_1\big(H^{p,q}_\C,g^{H^{p,q}_\C}\big) \;.
\end{split}
\end{align}
From \eqref{eq2-pf-prop-hodgeform} and \eqref{eq3-pf-prop-hodgeform},
we obtain \eqref{eq-prop-hodgeform}.
This completes the proof.
\end{proof}

For $r\in \N$,
we denote by $H[r]^\bullet$
the $r$-th right shift of $H^\bullet$,
i.e.,
\begin{equation}
\label{eq-def-shift}
H[r]_\Z^k = H^{k-2r}_\Z \;,\hspace{5mm}
H[r]_\C^{p,q} =  H^{p-r,q-r}_\C \;.
\end{equation}

\begin{prop}
\label{prop-shift-hodgeform}
The following identity holds,
\begin{equation}
\label{eq-prop-shift-hodgeform}
\omega_{H^\bullet} = \omega_{H[r]^\bullet} \;.
\end{equation}
\end{prop}
\begin{proof}
The right hand side of \eqref{eq-prop-hodgeform}
is invariant under right shift.
\end{proof}

\subsection{Quillen metric}
\label{subsect-quillen}

Let $X$ be a compact K{\"a}hler manifold of dimension $n$.
Let $E$ be a holomorphic vector bundle over $X$.
Let $\overline{\partial}^E$ be the Dolbeault operator on
\begin{equation}
\Omega^{0,\bullet}(X,E) =
\smooth(X,\Lambda^\bullet(\overline{T^*X})\otimes E) \;.
\end{equation}
For $q=0,\cdots,n$,
we have
\begin{equation}
H^q(X,E) = H^q\big(\Omega^{0,\bullet}(X,E),\overline{\partial}^E\big) \;.
\end{equation}
Set
\begin{equation}
\label{eq-lambda-E}
\lambda(E) = \det H^\bullet(X,E) = \bigotimes_{q=0}^n \big(\det H^q(X,E)\big)^{(-1)^q} \;.
\end{equation}

Let $g^{TX}$ be a K{\"a}hler metric on $TX$.
Let $g^E$ be a Hermitian metric on $E$.
Let $\big\langle\cdot,\cdot\big\rangle_{\Lambda^\bullet(\overline{T^*X})\otimes E}$
be the Hermitian product on $\Lambda^\bullet(\overline{T^*X})\otimes E$ induced by $g^{TX}$ and $g^E$.
Let $dv_X$ be the volume form on $X$ induced by $g^{TX}$.
For $s_1,s_2\in \Omega^{0,\bullet}(X,E)$,
set
\begin{equation}
\label{eq-def-L2-metric}
\big\langle s_1,s_2 \big\rangle = (2\pi)^{-n} \int_X
\big\langle s_1,s_2 \big\rangle_{\Lambda^\bullet(\overline{T^*X})\otimes E} dv_X \;.
\end{equation}

Let $\overline{\partial}^{E,*}$ be the formal adjoint of $\overline{\partial}^E$
with respect to the Hermitian product \eqref{eq-def-L2-metric}.
The Dolbeault Laplacian on $\Omega^{0,\bullet}(X,E)$ is defined by
\begin{equation}
\Delta^E = \overline{\partial}^E\overline{\partial}^{E,*} + \overline{\partial}^{E,*}\overline{\partial}^E \;.
\end{equation}
Then $\Delta^E$ preserves the degree,
i.e. $\Delta^E\big(\Omega^{0,q}(X,E)\big)\subseteq\Omega^{0,q}(X,E)$ for $q=0,\cdots,n$.
Let $\Delta^E_q$ be the restriction of $\Delta^E$ to $\Omega^{0,q}(X,E)$.
The operator $\Delta^E_q$ is essentially self-adjoint.
Its self-adjoint extension will still be denoted by $\Delta^E_q$.

By Hodge Theorem,
we have
\begin{equation}
\mathrm{Ker}\big(\Delta^E_q\big) =
\Big\{s\in\Omega^{0,q}(X,E)\;:\;
\overline{\partial}^Es=0\;,\;\overline{\partial}^{E,*}s=0\Big\} \;.
\end{equation}
Moreover,
the following map is bijective,
\begin{align}
\label{eq-iso-hodge}
\begin{split}
\mathrm{Ker}\big(\Delta^E_q\big) & \rightarrow H^q(X,E) \\
s & \mapsto [s] \;.
\end{split}
\end{align}
Let $\big|\cdot\big|_{\lambda(E)}$ be the metric on $\lambda(E)$
induced by the Hermitian product \eqref{eq-def-L2-metric}
via the isomorphism \eqref{eq-iso-hodge}.

Let $\mathrm{Sp}(\Delta^E_q)$ be the spectrum of $\Delta^E_q$.
For $s\in\C$ with $\mathrm{Re}(s)>n$,
set
\begin{equation}
\label{eq-def-theta}
\theta^E_X(s) = \sum_{q=1}^n (-1)^{q+1}q \sum_{\lambda\in\mathrm{Sp}(\Delta^E_q),\lambda\neq 0}  \lambda^{-s} \;.
\end{equation}
By \cite{se},
the function $\theta^E_X(s)$ extends to a meromorphic function of $s\in\C$,
which is holomorphic at $s=0$.

The Quillen metric on $\lambda(E)$ is defined by
\begin{equation}
\label{eq-def-quillen}
\big\lVert\cdot\big\rVert_{\lambda(E)} =
\exp\Big(\frac{1}{2}\frac{\partial \theta^E_X}{\partial s}(0)\Big)
\big|\cdot\big|_{\lambda(E)} \;.
\end{equation}

\noindent\textbf{Behavior under immersion.}
Let $X$ be a compact K{\"a}hler manifold of dimension $2$.
Let $Y\subseteq X$ be a curve.
Let $j:Y \rightarrow X$ be the canonical embedding.
Let $L$ be a holomorphic line bundle over $X$
together with $v \in H^0(X,L^{-1})$
such that $v: L \rightarrow \C$ provides a resolution of $j_*\mathscr{O}_Y$.
More precisely,
we have a short exact sequence of analytic coherent sheaves on $X$,
\begin{equation}
\label{eq-immersion-resol}
0 \rightarrow \mathscr{O}_X(L) \xrightarrow{v} \mathscr{O}_X \rightarrow j_*\mathscr{O}_Y \rightarrow 0 \;,
\end{equation}
where $\mathscr{O}_X \rightarrow j_*\mathscr{O}_Y$ is the obvious restriction map.
By the long exact sequence induced by \eqref{eq-immersion-resol},
there is a canonical section
\begin{equation}
\sigma \in \det H^\bullet(X,L) \otimes \big(\det H^{0,\bullet}(X)\big)^{-1} \otimes \det H^{0,\bullet}(Y) \;.
\end{equation}

Let $N_Y$ be the normal bundle of $Y\subseteq X$.
Let $\nabla^{L^{-1}}$ be a connection on $L^{-1}$.
Set
\begin{equation}
v' = \nabla^{L^{-1}} v \big|_Y \in H^0(Y,N_Y^{-1}\otimes L^{-1}) \;,
\end{equation}
which is independent of $\nabla^{L^{-1}}$.
We remark that $v'$ is nowhere vanishing.

Let $g^{TX}$ be a K{\"a}hler metric on $TX$.
Let $g^{TY}$ be the metric on $TY$ induced by $g^{TX}$.
Let $g^L$ be a Hermitian metric on $L$.
Let $\big\lVert\cdot\big\rVert_{\det H^\bullet(X,L)}$
be the Quillen metric on $\det H^\bullet(X,L)$ associated with $g^{TX}$ and $g^L$.
Let $\big\lVert\cdot\big\rVert_{\det H^{0,\bullet}(X)}$
be the Quillen metric on $\det H^{0,\bullet}(X)$ associated with $g^{TX}$.
Let $\big\lVert\cdot\big\rVert_{\det H^{0,\bullet}(Y)}$
be the Quillen metric on $\det H^{0,\bullet}(Y)$ associated with $g^{TY}$.
Let $\big\lVert\sigma\big\rVert$ be the norm of $\sigma$
with respect to the product metric.

Let $g^{N_Y}$ be the metric on $N_Y$ induced by $g^{TX}$.
Let $\big| v' \big|$ be the norm of $v'$
with respect to the metric on $N_Y^{-1}\otimes L^{-1}$
induced by $g^{N_Y}$ and $g^L$.
We assume that
\begin{equation}
\label{eq-A}
\big| v' \big| = 1 \;.
\end{equation}
The assumption \eqref{eq-A} is equivalent to assumption (A) \cite[Definition 1.5]{b90}.

Let $\big| v \big|$ be the norm of $v$
with respect to the metric on $L^{-1}$ induced by $g^L$.

Let $\zeta(s)$ be the Riemann zeta function.

\begin{thm}
\label{thm-quillen-immersion}
The following identity holds,
\begin{align}
\label{eq-thm-quillen-immersion}
\begin{split}
\log \big\lVert\sigma\big\rVert^2 & =
- \int_X \Big(1+\frac{c_1}{2}+\frac{c_1^2+c_2}{12}\Big)(TX,g^{TX})
\Big(1+\frac{c_1}{2}+\frac{c_1^2}{6}\Big)(L,g^L) \log \big|v\big|^2 \\
& \hspace{5mm} + \frac{1}{12} \int_Y \widetilde{c}(TY,TX\big|_Y,g^{TX}\big|_Y) \\
& \hspace{5mm} + \big(\zeta(-1) + 2\zeta'(-1)\big)
\Big( \int_X c_1(TX)c_1(L) + \int_Y c_1(TY) \Big) \;.
\end{split}
\end{align}
\end{thm}
\begin{proof}
We view $\C$ as a trivial line bundle over $X$.
Let $g^\C$ be the obvious metric on $\C$.
Let $T(v,g^L)\in A^{\bullet,\bullet}(X)$ be
the Bott-Chern current of the complex of holomorphic vector bundles $v: L \rightarrow \C$
equipped with metrics $g^L,g^\C$.
By \cite[page 285]{bgs90} and \cite[Theorem 3.17]{bgs90},
we have
\begin{equation}
\label{eq1-pf-thm-quillen-immersion}
T(v,g^L) = \mathrm{Td}^{-1}(L^{-1},g^{L^{-1}}) \log \big|v\big|^2
+ \partial \alpha + \overline{\partial} \beta \;,
\end{equation}
where $\alpha,\beta\in A^{\bullet,\bullet}(X)$.

Let $\widetilde{\mathrm{Td}}(TY,TX\big|_Y,g^{TX}\big|_Y)\in A^{\bullet,\bullet}(Y)$ be as in \cite[(0.3)]{ble}.
By \cite[(0.3)]{ble}, \eqref{eq-def-BC}
and the uniqueness of Bott-Chern form (see \cite[Theorem 1.29]{bgs1}),
we have
\begin{align}
\label{eq2-pf-thm-quillen-immersion}
\begin{split}
\Big\{\widetilde{\mathrm{Td}}(TY,TX\big|_Y,g^{TX}\big|_Y)\Big\}^{(0,0)} & = 0 \;,\\
\Big\{\widetilde{\mathrm{Td}}(TY,TX\big|_Y,g^{TX}\big|_Y)\Big\}^{(1,1)} & =
\frac{1}{12} \Big\{\widetilde{c}(TY,TX\big|_Y,g^{TX}\big|_Y)\Big\}^{(1,1)} + \partial \alpha + \overline{\partial} \beta \;,
\end{split}
\end{align}
where $\alpha,\beta\in\Omega^{\bullet,\bullet}(Y)$.

From \cite[Theorem 0.1]{ble},
\eqref{eq1-pf-thm-quillen-immersion} and \eqref{eq2-pf-thm-quillen-immersion},
we obtain \eqref{eq-thm-quillen-immersion}.
This completes the proof.
\end{proof}

\noindent\textbf{Behavior under blow-up.}
Let $X$ and $W$ be compact K{\"a}hler manifolds of dimension $n\geqslant 2$.
Let $Z$ be a compact K{\"a}hler manifold of dimension $n-2$.
Let $i: Z\rightarrow X$ and $j: Z\rightarrow W$ be complex immersions.
We assume that there exist open subsets
\begin{equation}
i(Z) \subseteq U \subseteq X \;,\hspace{5mm}
j(Z) \subseteq \mathcal{U} \subseteq W
\end{equation}
and a biholomorphic map $\varphi: U \rightarrow \mathcal{U}$
such that the following diagram commutes,
\begin{equation}
\xymatrix{
Z \ar[d]_i \ar[dr]^j & \\
U \ar[r]^{\hspace{-1.5mm}\varphi} & \mathcal{U} \;.
}
\end{equation}

Let $E$ be a holomorphic vector bundle over $X$.
Let $F$ be a holomorphic vector bundle over $W$.
We assume that there is an isomorphism
$\phi: E\big|_U \rightarrow F\big|_\mathcal{U}$.

Let $f: X' \rightarrow X$ be the blow-up along $i(Z)$.
Then $\det H^\bullet(X,E)$ is canonically isomorphic to $\det H^\bullet(X',f^*E)$.
Let
\begin{equation}
\label{eq-tauEf}
\tau_E^f \in \big(\det H^\bullet(X,E)\big)^{-1} \otimes \det H^\bullet(X',f^*E)
\end{equation}
be the section which defines the canonical isomorphism.
Let $g: W' \rightarrow W$ be the blow-up along $j(Z)$.
We construct
\begin{equation}
\label{eq-tauFg}
\tau_F^g \in \big(\det H^\bullet(W,F)\big)^{-1} \otimes \det H^\bullet(W',g^*F)
\end{equation}
in the same way.

Let $g^{TX}$ and $g^{TW}$ be K{\"a}hler metrics on $TX$ and $TW$.
We assume that
\begin{equation}
\varphi_*\big(g^{TX}\big|_U\big)  = g^{TW}\big|_{\mathcal{U}} \;.
\end{equation}
We denote $U' = f^{-1}(U)$ and $\mathcal{U}' = g^{-1}(\mathcal{U})$.
Let $\varphi': U'\rightarrow \mathcal{U}'$ be the lift of $\varphi$.
Let $g^{TX'}$ and $g^{TW'}$ be K{\"a}hler metrics on $TX'$ and $TW'$ such that
$X'\backslash U' \xrightarrow{f} X \backslash U$,
$W'\backslash \mathcal{U}' \xrightarrow{g} W \backslash \mathcal{U}$
and $U' \xrightarrow{\varphi'} \mathcal{U}'$
are isometric.

Let $g^E$ and $g^F$ be Hermitian metrics on $E$ and $F$ such that
$E\big|_U \xrightarrow{\phi} F\big|_\mathcal{U}$
is isometric.

Let $\big\lVert\tau_E^f \big\rVert$ and $\big\lVert\tau_F^g \big\rVert$
be the norms of $\tau_E^f$ and $\tau_F^g$ with respect to the Quillen metrics.

\begin{thm}
\label{thm-quillen-bl}
The following identity holds,
\begin{equation}
\big\lVert\tau_E^f \big\rVert = \big\lVert\tau_F^g \big\rVert \;.
\end{equation}
\end{thm}
\begin{proof}
This is an immediate consequence of \cite[Theorem 8.10]{b97}.
\end{proof}

Let $N$ be the normal bundle of $i(Z)\subseteq X$.
Set $D=\mathbb{P}(N)$.
We have canonical isomorphisms
$i_D: D \rightarrow f^{-1}\big(i(Z)\big)$ and
$j_D: D \rightarrow g^{-1}\big(j(Z)\big)$.

For $p=1,\cdots,n$,
there exists a holomorphic vector bundle $G^p$ over $D$
together with holomorphic maps
\begin{equation}
r: \Lambda^p(T^*X')\big|_{i_D(D)} \rightarrow i_{D,*}G^p \;,\hspace{5mm}
s: \Lambda^p(T^*W')\big|_{j_D(D)} \rightarrow j_{D,*}G^p
\end{equation}
such that
\begin{align}
\label{eq-ses-bl}
\begin{split}
0 \rightarrow f^*\Omega^p_X \rightarrow \Omega^p_{X'}
\xrightarrow{r} i_{D,*}\mathscr{O}_D(G^p) \rightarrow 0 \;,\\
0 \rightarrow g^*\Omega^p_W \rightarrow \Omega^p_{W'}
\xrightarrow{s} j_{D,*}\mathscr{O}_D(G^p) \rightarrow 0
\end{split}
\end{align}
are exact.
Let
\begin{align}
\begin{split}
& \eta^f_p \in
\Big(\det H^\bullet\big(X',f^*\Lambda^p(T^*X)\big)\Big)^{-1} \otimes
\det H^{p,\bullet}(X') \otimes
\big(\det H^\bullet(D,G^p)\big)^{-1} \;,\\
& \eta^g_p \in
\Big(\det H^\bullet\big(W',g^*\Lambda^p(T^*W)\big)\Big)^{-1} \otimes
\det H^{p,\bullet}(W') \otimes
\big(\det H^\bullet(D,G^p)\big)^{-1}
\end{split}
\end{align}
be the canonical sections induced by \eqref{eq-ses-bl}.
Let
\begin{align}
\begin{split}
& \tau_p^f \in \big(\det H^{p,\bullet}(X)\big)^{-1} \otimes \det H^\bullet\big(X',f^*\Lambda^p(T^*X)\big) \;, \\
& \tau_p^g \in \big(\det H^{p,\bullet}(W)\big)^{-1} \otimes \det H^\bullet\big(W',g^*\Lambda^p(T^*W)\big)
\end{split}
\end{align}
be as in \eqref{eq-tauEf}, \eqref{eq-tauFg}
with $E = \Lambda^p(T^*X)$, $F = \Lambda^p(T^*W)$.
Set
\begin{align}
\label{eq-sigma-fg-p}
\begin{split}
& \sigma^f_p = \tau^f_p\otimes\eta^f_p \in
\big(\det H^{p,\bullet}(X)\big)^{-1} \otimes
\det H^{p,\bullet}(X') \otimes
\big(\det H^\bullet(D,G^p)\big)^{-1} \;,\\
& \sigma^g_p = \tau^g_p\otimes\eta^g_p \in
\big(\det H^{p,\bullet}(W)\big)^{-1} \otimes
\det H^{p,\bullet}(W') \otimes
\big(\det H^\bullet(D,G^p)\big)^{-1} \;.
\end{split}
\end{align}

Let $g^{TD}$ be the K{\"a}hler metric on $TD$
induced by $g^{TX'}$ via the embedding $i_D: D \rightarrow X'$.
Let $g^{G^p}$ be the Hermitian metric on $G^p$
induced by $g^{TX'}$ via $r: \Lambda^p(T^*X')\big|_{i_D(D)} \rightarrow i_{D,*}G^p$.

Let $\big\lVert\sigma^f_p\big\rVert$ and $\big\lVert\sigma^g_p\big\rVert$
be the norms of $\sigma^f_p$ and $\sigma^g_p$ with respect to the Quillen metrics.

\begin{cor}
\label{cor-quillen-bl}
The following identity holds,
\begin{equation}
\label{eq-cor-quillen-bl}
\big\lVert\sigma^f_p\big\rVert = \big\lVert\sigma^g_p\big\rVert \;.
\end{equation}
\end{cor}
\begin{proof}
In the whole proof,
$\big\lVert\cdot\big\rVert$ is always the Quillen metric.
We have
\begin{equation}
\label{eq1-pf-cor-quillen-bl}
\big\lVert\sigma^f_p\big\rVert = \big\lVert\tau^f_p\big\rVert\big\lVert\eta^f_p\big\rVert \;,\hspace{5mm}
\big\lVert\sigma^g_p\big\rVert = \big\lVert\tau^g_p\big\rVert\big\lVert\eta^g_p\big\rVert \;.
\end{equation}
By Theorem \ref{thm-quillen-bl},
we have
\begin{equation}
\label{eq2-pf-cor-quillen-bl}
\big\lVert\tau^f_p\big\rVert = \big\lVert\tau^g_p\big\rVert \;.
\end{equation}
By \cite[Theorem 0.1]{ble},
we have
\begin{equation}
\label{eq3-pf-cor-quillen-bl}
\big\lVert\eta^f_p\big\rVert = \big\lVert\eta^g_p\big\rVert \;.
\end{equation}
From \eqref{eq1-pf-cor-quillen-bl}-\eqref{eq3-pf-cor-quillen-bl},
we obtain \eqref{eq-cor-quillen-bl}.
This completes the proof.
\end{proof}

\subsection{BCOV torsion}
\label{subsect-bcov}

Let $X$ be a compact K{\"a}hler manifold of dimension $n$.

For $p=0,\cdots,n$, set
\begin{equation}
\label{eq-def-lambda-p}
\lambda_p(X) = \det H^{p,\bullet}(X)
= \bigotimes_{q=0}^n \Big( \det H^{p,q}(X) \Big)^{(-1)^q} \;.
\end{equation}
Set
\begin{equation}
\label{eq-def-lambda}
\lambda(X)
= \bigotimes_{p=1}^n \Big(\lambda_p(X)\Big)^{(-1)^pp}
= \bigotimes_{0\leqslant p,q\leqslant n} \Big( \det H^{p,q}(X) \Big)^{(-1)^{p+q}p} \;.
\end{equation}
Set
\begin{equation}
\label{eq-def-lambda-dR}
\lambda_\mathrm{dR}(X)
= \bigotimes_{k=1}^{2n} \Big(\det H^k_\mathrm{dR}(X)\Big)^{(-1)^kk}
=  \lambda(X) \otimes \overline{\lambda(X)}\;.
\end{equation}

For $\mathbb{A} = \Z,\R,\C$,
we denote by $H^\bullet_\mathrm{Sing}(X,\mathbb{A})$
the singular cohomology of $X$ with coefficients in $\mathbb{A}$.
For $k=0,\cdots,n$,
let
\begin{equation}
\sigma_{k,1},\cdots,\sigma_{k,b_k} \in
\mathrm{Im}\Big(H^k_\mathrm{Sing}(X,\Z)\rightarrow H^k_\mathrm{Sing}(X,\R)\Big)
\end{equation}
be a basis of the lattice.
We identify $H^k_\mathrm{dR}(X)$ with $H^k_\mathrm{Sing}(X,\C)$ as follows,
\begin{align}
\begin{split}
H^k_\mathrm{dR}(X) & \rightarrow H^k_\mathrm{Sing}(X,\C) \\
[\alpha] & \mapsto \Big[\mathfrak{a} \mapsto \int_\mathfrak{a}\alpha\Big]\;,
\end{split}
\end{align}
where $\alpha$ is a closed $k$-form on $X$
and $\mathfrak{a}$ is a $k$-chain in $X$.
Then $\sigma_{k,1},\cdots,\sigma_{k,b_k}$ may be viewed as elements in $H^k_\mathrm{dR}(X)$.
Set
\begin{equation}
\label{eq-def-sigma}
\sigma = \bigotimes_{k=1}^{2n} \big( \sigma_{k,1}\wedge\cdots\wedge\sigma_{k,b_k} \big)^{(-1)^kk} \in \lambda_\mathrm{dR}(X) \;,
\end{equation}
which is well-defined up to $\pm 1$.

Let $\omega$ be a K{\"a}hler form on $X$.

Let $\big\lVert\cdot\big\rVert_{\lambda_p(X),\omega}$
be the Quillen metric on $\lambda_p(X)$ associated with $\omega$.
Let \hbox{$\big\lVert\cdot\big\rVert_{\lambda(X),\omega}$}
be the metric on $\lambda(X)$
induced by \hbox{$\big\lVert\cdot\big\rVert_{\lambda_p(X),\omega}$}
via \eqref{eq-def-lambda}.
Let \hbox{$\big\lVert\cdot\big\rVert_{\lambda_\mathrm{dR}(X),\omega}$}
be the metric on $\lambda_\mathrm{dR}(X)$
induced by \hbox{$\big\lVert\cdot\big\rVert_{\lambda(X),\omega}$}
via \eqref{eq-def-lambda-dR}.
We define
\begin{equation}
\label{eq-def-bcov-torsion}
\tau_\mathrm{BCOV}(X,\omega) =
\log \big\lVert\sigma\big\rVert_{\lambda_\mathrm{dR}(X),\omega} \;.
\end{equation}

Let $\pi_\mathscr{X}: \mathscr{X}\rightarrow S$ be a locally K{\"a}hler holomorphic fibration.
For $s\in S$,
we denote $X_s = \pi_\mathscr{X}^{-1}(s)$.
We assume that $\dim X_s = n$ for any $s\in S$.

Let $H^\bullet(X)$ be the variation of Hodge structure
associated with $\pi_\mathscr{X}: \mathscr{X}\rightarrow S$,
i.e.,
\begin{align}
\begin{split}
H^k_\Z(X) & = H^k_\mathrm{Sing}(X,\Z) \;,\\
H^k_\C(X) & = H^k_\mathrm{Sing}(X,\C) = H^k_\mathrm{dR}(X) \;,\hspace{5mm}
F^r H^k_\C(X) = \bigoplus_{p=r}^k H^{p,k-p}(X) \;.
\end{split}
\end{align}
Let $\omega_{H^\bullet(X)}\in\Omega^{1,1}(S)$
be the Hodge form (see \eqref{eq-def-hodgeform}) of $H^\bullet(X)$.

Let $\omega\in\Omega^{1,1}(\mathscr{X})$ be a fiberwise K{\"a}hler form,
i.e., $\omega\big|_{X_s}$ is a K{\"a}hler form for any $s\in S$.
Let $TX \subseteq T\mathscr{X}$ be the relative tangent bundle.
Let $g^{TX}$ be the metric on $TX$ induced by $\omega$.
We denote by $\tau_\mathrm{BCOV}(X,\omega)$
the function $s\mapsto \tau_\mathrm{BCOV}\big(X_s,\omega\big|_{X_s}\big)$ on $S$.

\begin{thm}
\label{thm-anomaly}
The following identity holds,
\begin{equation}
\label{eq-anomaly}
\frac{\overline{\partial}\partial}{2\pi i} \tau_\mathrm{BCOV}(X,\omega)
= \omega_{H^\bullet(X)} + \frac{1}{12}
\int_X c_1\big(TX,g^{TX}\big)c_n\big(TX,g^{TX}\big) \;.
\end{equation}
\end{thm}
\begin{proof}
Let $\big(H^{p,q}(X)\big)_{0\leqslant p,q\leqslant n}$ be the fiberwise Dolbeault cohomologies,
which are holomorphic vector bundles over $S$.
Let $\lambda_p(X)$ and $\lambda(X)$ be as in \eqref{eq-def-lambda-p} and \eqref{eq-def-lambda},
which are holomorphic line bundles over $S$.
Let $\big\lVert\cdot\big\rVert_{\lambda_p(X),\omega}$
be the metric on $\lambda_p(X)$ such that
its restriction to $\lambda_p(X_s)$ is the Quillen metric associated with $\omega\big|_{X_s}$.
Let $\big\lVert\cdot\big\rVert_{\lambda(X),\omega}$
be the metric on $\lambda(X)$ induced by $\big\lVert\cdot\big\rVert_{\lambda_p(X),\omega}$.

We may assume that $S$ a unit polydisc in $\C^k$,
where $k = \dim S$.
Let $\tau\in H^0(S,\lambda(X))$ and $f\in\smooth(S,\C)$
be as in \eqref{eq-def-f}.
By the Poincar{\'e}-Lelong formula,
 \eqref{eq-def-f}, \eqref{eq-def-hodgeform} and \eqref{eq-def-bcov-torsion},
we have
\begin{align}
\label{eq1-pf-thm-anomaly}
\begin{split}
\frac{\overline{\partial}\partial}{2\pi i} \tau_\mathrm{BCOV}(X,\omega)
& = \frac{\overline{\partial}\partial \mathrm{Re} f}{2\pi i}
+ \frac{\overline{\partial}\partial \log \big\lVert\tau\big\rVert^2_{\lambda(X),\omega}}{2\pi i} \\
& = \omega_{H^\bullet(X)}
- c_1\big(\lambda(X),\big\lVert\cdot\big\rVert_{\lambda(X),\omega}\big) \;.
\end{split}
\end{align}
Let $g^{\Lambda^\bullet(T^*X)}$ be
the Hermitian metric on $\Lambda^\bullet(T^*X)$ induced by $g^{TX}$.
By \cite[Theorem 0.1]{bgs1}, \cite[page 374]{bcov2} and \eqref{eq-def-lambda},
we have
\begin{align}
\label{eq2-pf-thm-anomaly}
\begin{split}
& c_1\big(\lambda(X),\big\lVert\cdot\big\rVert_{\lambda(X),\omega}\big) \\
& = \sum_{p=1}^n (-1)^pp
c_1\big(\lambda_p(X),\big\lVert\cdot\big\rVert_{\lambda_p(X),\omega}\big) \\
& = \sum_{p=1}^n (-1)^pp
\bigg\{ \int_X \mathrm{Td}\big(TX,g^{TX}\big)
\mathrm{ch}\big(\Lambda^p(T^*X),g^{\Lambda^p(T^*X)}\big) \bigg\}^{(1,1)} \\
& = - \frac{1}{12}
\int_X c_1\big(TX,g^{TX}\big)c_n\big(TX,g^{TX}\big) \;.
\end{split}
\end{align}
From \eqref{eq1-pf-thm-anomaly}
and \eqref{eq2-pf-thm-anomaly},
we obtain \eqref{eq-anomaly}.
This completes the proof.
\end{proof}

\section{BCOV invariant $\tau(X,Y)$}
\label{sect-bcov}

\subsection{Construction of $\tau(X,Y)$ and proof of Theorem \ref{intro-thm-curvature}}
\label{subsect-construction}

Let $X$ be a compact K{\"a}hler manifold of dimension $n$.
Let $K_X$ be its canonical bundle.
Let $m\in\Z\backslash\{0,-1\}$.
Let $K_X^m$ be the $m$-th tensor power of $K_X$.
We assume that $H^0(X,K_X^m)\neq 0$.
Let
\begin{equation}
\gamma \in H^0(X,K_X^m)\backslash\{0\} \;.
\end{equation}
Set
\begin{equation}
Y = \mathrm{Div}(\gamma) \;.
\end{equation}
We assume that $Y$ is smooth and reduced.
We will call $(X,Y)$ an $m$-Calabi-Yau pair.

Let $\nabla^{K_X^m}$ be a connection on $K_X^m$.
Let $N_Y$ be the normal bundle of $Y \subseteq X$.
Set
\begin{equation}
\label{eq-gammaprim}
\gamma' = \nabla^{K_X^m} \gamma \big|_Y
\in H^0(Y,N^{-1}_Y \otimes K_X^m)
= H^0(Y,N^{-m-1}_Y \otimes K_Y^m) \;,
\end{equation}
which is independent of $\nabla^{K_X^m}$.

Let $\omega$ be a K{\"a}hler form on $X$.

Let $g^{TX}$ be the metric on $TX$ induced by $\omega$.
Let $\big|\cdot\big|$ be the norm on $K_X^m$ induced by $g^{TX}$.
Set
\begin{equation}
\label{eq-def-aX}
a_X(\gamma,\omega)
=  \frac{1}{12} \int_X c_n\big(TX,g^{TX}\big) \log \big|\gamma\big|^2 \;.
\end{equation}
Let $g^{TY}$ be the metric on $TY$ induced by $g^{TX}$.
Let $\big|\cdot\big|$ be the norm
on $N^{-m-1}_Y \otimes K_Y^m$ induced by $g^{TX}$.
Set
\begin{equation}
\label{eq-def-aY}
a_Y(\gamma,\omega)
= \frac{1}{12}
\int_Y c_{n-1}\big(TY,g^{TY}\big) \log \big|\gamma'\big|^2 \;.
\end{equation}
Let $\widetilde{c}(\cdot,\cdot,\cdot)$ be as in \eqref{eq-def-BC}.
Set
\begin{equation}
\label{eq-def-bY}
b_Y(\omega) = \frac{1}{12} \int_Y \widetilde{c}\big(TY,TX\big|_Y,g^{TX}\big|_Y\big) \;.
\end{equation}
Let $\tau_\mathrm{BCOV}(\cdot,\cdot)$ be as in \eqref{eq-def-bcov-torsion}.
Set
\begin{align}
\label{eq-def-tau-omega}
\begin{split}
\tau(X,\gamma,\omega) =
& \; \tau_\mathrm{BCOV}(X,\omega)
- \frac{1}{m} a_X(\gamma,\omega) \\
& \;  - \frac{1}{m+1} \tau_\mathrm{BCOV}(Y,\omega|_Y)
+ \frac{1}{m(m+1)} a_Y(\gamma,\omega)
+ \frac{1}{m}b_Y(\omega) \;.
\end{split}
\end{align}

\begin{thm}
\label{thm-ind-omega}
The real number $\tau(X,\gamma,\omega)$ is independent of $\omega$.
\end{thm}
\begin{proof}
Let $\mu_1$ and $\mu_2$ be two K{\"a}hler forms on $X$.
Let $\big(\omega_s\big)_{s\in{\C}P^1}$ be a smooth family of K{\"a}hler forms on $X$
such that $\omega_0 = \mu_1$ and $\omega_\infty = \mu_2$.

Let $\tau(X,\gamma,\omega)$
be the function $s\mapsto \tau(X,\gamma,\omega_s)$ on ${\C}P^1$.
In the same way,
we have functions
$\tau_\mathrm{BCOV}(X,\omega)$, $\tau_\mathrm{BCOV}(Y,\omega|_Y)$,
$a_X(\gamma,\omega)$, $a_Y(\gamma,\omega)$ and $b_Y(\omega)$ on ${\C}P^1$.

We view $TX$ (resp. $TY$) as a holomorphic vector bundle over $X\times {\C}P^1$ (resp. $Y\times {\C}P^1$).
Let $g^{TX}$ be the metric on $TX$ induced by $\big(\omega_s\big)_{s\in{\C}P^1}$,
i.e., $g^{TX}\big|_{X\times\{s\}}$ is induced by $\omega_s$.
Let $g^{TY}$ be the metric on $TY$ induced by $g^{TX}$.
By \eqref{eq-anomaly},
we have
\begin{align}
\label{eq-anomaly-XY}
\begin{split}
\frac{\overline{\partial}\partial}{2\pi i}
\tau_\mathrm{BCOV}(X,\omega)
& = \frac{1}{12}
\int_X c_1\big(TX,g^{TX}\big)c_n\big(TX,g^{TX}\big) \;,\\
\frac{\overline{\partial}\partial}{2\pi i}
 \tau_\mathrm{BCOV}(Y,\omega|_Y)
& = \frac{1}{12}
\int_Y c_1\big(TY,g^{TY}\big)c_{n-1}\big(TY,g^{TY}\big) \;.
\end{split}
\end{align}

We view $N_Y$ as a holomorphic line bundle over $Y\times {\C}P^1$.
Let $g^{N_Y}$ be the metric on $N_Y$ induced by $g^{TX}$.
By the Poincar{\'e}-Lelong formula,
\eqref{eq-def-aX} and \eqref{eq-def-aY},
we have
\begin{align}
\label{eq-bc1}
\begin{split}
\frac{\overline{\partial}\partial}{2\pi i} a_X(\gamma,\omega) & =
\frac{m}{12} \int_X c_1\big(TX,g^{TX}\big)c_n\big(TX,g^{TX}\big)
+ \frac{1}{12} \int_Y c_n\big(TX,g^{TX}\big) \;,\\
\frac{\overline{\partial}\partial}{2\pi i} a_Y(\gamma,\omega) & =
\frac{1}{12} \int_Y
\Big( m c_1\big(TY,g^{TY}\big) + (m+1) c_1\big(N_Y,g^{N_Y}\big) \Big)
c_{n-1}\big(TY,g^{TY}\big) \;.
\end{split}
\end{align}

By \eqref{eq-def-BC} and \eqref{eq-def-bY},
we have
\begin{equation}
\label{eq-bc2}
\frac{\overline{\partial}\partial}{2 \pi i} b_Y(\omega) =
\frac{1}{12} \int_Y c_n\big(TX,g^{TX}\big)
- \frac{1}{12} \int_Y c_1\big(TY,g^{TY}\big)c_{n-1}\big(N_Y,g^{N_Y}\big) \;.
\end{equation}

From \eqref{eq-def-tau-omega}-\eqref{eq-bc2},
we obtain
\begin{equation}
\overline{\partial}\partial\tau(X,\gamma,\omega) = 0 \;.
\end{equation}
Hence $\tau(X,\gamma,\omega)$ is a constant function on ${\C}P^1$.
We get $\tau(X,\gamma,\mu_1) = \tau(X,\gamma,\mu_2)$.
This completes the proof.
\end{proof}

Recall that
$\big|\cdot\big|^{1/m}: \smooth(X,K_X^m\otimes\overline{K_X^m})\rightarrow\Omega^{n,n}(X)$
was defined in the paragraph containing \eqref{eq0-intro-wp}.
For $z\in\C^*$,
a direct calculation yields
\begin{align}
\label{eq-mult-z}
\begin{split}
& a_X(z\gamma,\omega) - a_X(\gamma,\omega)
= \frac{\chi(X)}{12} \log |z|^2 \;,\\
& a_Y(z\gamma,\omega) - a_Y(\gamma,\omega)
= \frac{\chi(Y)}{12} \log |z|^2 \;, \\
& \log \int_X \big|z\gamma\overline{z\gamma}\big|^{1/m}
- \log \int_X \big|\gamma\overline{\gamma}\big|^{1/m}
= \frac{1}{m} \log |z|^2 \;.
\end{split}
\end{align}
Recall that $w_m(X)$ was defined by \eqref{eq-def-w}.
By \eqref{eq-def-w}, \eqref{eq-def-tau-omega} and \eqref{eq-mult-z},
we have
\begin{align}
\label{eq-ind-mult}
\begin{split}
& \tau(X,z\gamma,\omega)
+ w_m(X) \log \int_X \big|z\gamma\overline{z\gamma}\big|^{1/m} \\
& = \tau(X,\gamma,\omega)
+ w_m(X) \log \int_X \big|\gamma\overline{\gamma}\big|^{1/m}\;.
\end{split}
\end{align}

\begin{defn}
\label{def-tau}
The BCOV invariant of $(X,Y)$ is defined by
\begin{equation}
\label{eq-def-tau}
\tau(X,Y) = \tau(X,\gamma,\omega)
+ w_m(X) \log \int_X \big|\gamma\overline{\gamma}\big|^{1/m} \;.
\end{equation}
By Theorem \ref{thm-ind-omega} and \eqref{eq-ind-mult},
$\tau(X,Y)$ is well-defined.
\end{defn}

\begin{proof}[Proof of Theorem \ref{intro-thm-curvature}]
We may assume that $S$ is a unit polydisc in $\C^k$,
where $k = \dim S$.
Let $\omega\in\Omega^{1,1}(\mathscr{X})$ be a fiberwise K{\"a}hler form.
Let $g^{TX}$ (resp. $g^{TY}$)
be the metric on $TX$ (resp. $TY$)
induced by $\omega$.
Let $\tau_\mathrm{BCOV}(X,\omega)$
(resp. $\tau_\mathrm{BCOV}(Y,\omega|_Y)$)
be the function $s\mapsto\tau_\mathrm{BCOV}(X_s,\omega|_{X_s})$
(resp. $s\mapsto\tau_\mathrm{BCOV}(Y_s,\omega|_{Y_s})$)
on $S$.
By \eqref{eq-anomaly},
we have
\begin{align}
\label{eq1-pf-thm-curvature}
\begin{split}
\frac{\overline{\partial}\partial}{2\pi i} \tau_\mathrm{BCOV}(X,\omega)
& = \omega_{H^\bullet(X)}
+ \frac{1}{12}
\int_X c_1\big(TX,g^{TX}\big)c_n\big(TX,g^{TX}\big) \;,\\
\frac{\overline{\partial}\partial}{2\pi i} \tau_\mathrm{BCOV}(Y,\omega|_Y)
& = \omega_{H^\bullet(Y)}
+ \frac{1}{12}
\int_Y c_1\big(TY,g^{TY}\big)c_{n-1}\big(TY,g^{TY}\big) \;.
\end{split}
\end{align}
Noticing that \eqref{eq-bc1} and \eqref{eq-bc2} equally hold for
$\big(\pi_\mathscr{X}: \mathscr{X} \rightarrow S, \pi_\mathscr{Y}: \mathscr{Y} \rightarrow S\big)$,
formula \eqref{eq-intro-thm-curvature} follows from
\eqref{eq-bc1},
\eqref{eq-bc2},
\eqref{eq1-pf-thm-curvature} and
the definition of Weil-Petersson form
(see \eqref{eq0-intro-wp}, \eqref{eq-intro-wp}).
This completes the proof.
\end{proof}

\subsection{An auxiliary invariant}

Let $X$ be a compact K{\"a}hler manifold of dimension $n$.
Let $K_X$ be its canonical bundle.
We assume that $\mathscr{M}(X,K_X) \neq 0$.
Let
\begin{equation}
\gamma \in \mathscr{M}(X,K_X) \backslash\{0\} \;.
\end{equation}
We assume that
\begin{equation}
\label{eq-div-gammamero}
\mathrm{Div}(\gamma) = Y_1 - l Y_2
\end{equation}
with $Y_1,Y_2\subseteq X$ smooth and reduced,
$Y_1 \cap Y_2 = \emptyset$ and $l\geqslant 2$.

For $k\in\N$,
we denote by $\big(T^*X\oplus\overline{T^*X}\big)^{\otimes k}$
the $k$-th tensor power of $T^*X\oplus\overline{T^*X}$.
We denote $E_k = \big(T^*X\oplus\overline{T^*X}\big)^{\otimes k} \otimes K_X^{-1}$.
Let $\nabla^{E_k}$ be a connection on $E_k$.
Let
\begin{equation}
\gamma^{-1} \in \mathscr{M}(X,K_X^{-1}) = \mathscr{M}(X,E_0)
\end{equation}
be the inverse of $\gamma$.
Let $N_{Y_2}$ be the normal bundle of $Y_2\subseteq X$.
Set
\begin{equation}
\gamma'_2 = \nabla^{E_{l-1}}\cdots\nabla^{E_0}\gamma^{-1}\big|_{Y_2}
\in H^0(Y_2,N_{Y_2}^{-l} \otimes K_X^{-1}) = H^0(Y_2,N_{Y_2}^{-l+1} \otimes K_{Y_2}^{-1}) \;,
\end{equation}
which is independent of $\big(\nabla^{E_k}\big)_{k=0,\cdots,l-1}$.

Let $\omega$ be a K{\"a}hler form on $X$.

Let $g^{TY_2}$ be the metric on $TY_2$ induced by $\omega$.
Let $\big|\cdot\big|$ be the norm on $N_{Y_2}^{-l+1}\otimes K_{Y_2}^{-1}$
induced by $\omega$.
Set
\begin{equation}
a_{Y_2}(\gamma,\omega) = \frac{1}{12} \int_{Y_2} c_{n-1}\big(TY_2,g^{TY_2}\big)
\log \big|\gamma'_2\big|^2 \;.
\end{equation}
Let $a_X(\gamma,\omega)$ be as in \eqref{eq-def-aX}.
Let $a_{Y_1}(\gamma,\omega)$ be as in \eqref{eq-def-aY} with $Y$ replaced by $Y_1$.
Let $b_{Y_1}(\omega)$ (resp. $b_{Y_2}(\omega)$) be as in \eqref{eq-def-bY} with $Y$ replaced by $Y_1$ (resp. $Y_2$).
Set
\begin{align}
\label{eq-def-tau-omega-aux}
\begin{split}
\tau(X,\gamma,\omega) & =
\tau_\mathrm{BCOV}(X,\omega)
- a_X(\gamma,\omega) \\
& \hspace{5mm}
- \frac{1}{2} \tau_\mathrm{BCOV}(Y_1,\omega|_{Y_1})
+ \frac{1}{2} a_{Y_1}(\gamma,\omega)
+ b_{Y_1}(\omega) \\
& \hspace{5mm}
- \frac{l}{l-1} \tau_\mathrm{BCOV}(Y_2,\omega|_{Y_2})
- \frac{l}{l-1} a_{Y_2}(\gamma,\omega)
- l b_{Y_2}(\omega) \;.
\end{split}
\end{align}
In particular,
if $\gamma\in H^0(X,K_X)$,
then \eqref{eq-def-tau-omega-aux} is compatible with \eqref{eq-def-tau-omega}.

\begin{thm}
\label{thm-ind-omega-aux}
The real number $\tau(X,\gamma,\omega)$ is independent of $\omega$.
\end{thm}
\begin{proof}
We proceed in the same way as in the proof of Theorem \ref{thm-ind-omega}.
\end{proof}

Set
\begin{equation}
\label{eq-def-tau-aux}
\tau(X,\gamma) = \tau(X,\gamma,\omega) \;.
\end{equation}
By Theorem \ref{thm-ind-omega-aux},
$\tau(X,\gamma)$ is well-defined.

Set
\begin{equation}
\label{eq-def-w-gamma}
w(X,\gamma) = \frac{\chi(X)}{12} - \frac{\chi(Y_1)}{24} - \frac{l\chi(Y_2)}{12(l-1)} \;.
\end{equation}
In particular,
if $\gamma\in H^0(X,K_X)$,
then
\begin{equation}
\label{eq-w-w1}
w(X,\gamma) = w_1(X) \;.
\end{equation}

For $z\in\C^*$,
the same calculation as in \eqref{eq-mult-z} yields
\begin{equation}
\label{eq-z}
\tau(X,z\gamma) - \tau(X,\gamma) =
- w(X,\gamma) \log |z|^2 \;.
\end{equation}

\begin{rem}
\label{rem-tau-aux-bl}
Let $Z\subseteq X$ be a closed complex submanifold of codimension $2$
such that $Z\cap(Y_1 \cup Y_2) = \emptyset$.
Let $f: X' \rightarrow X$ be the blow-up along $Z$.
Let $Y_1$, $Y_2$ and $l$ be as in \eqref{eq-div-gammamero}.
Set $Y_1' = f^{-1}(Z \cup Y_1)$ and $Y_2' = f^{-1}(Y_2)$.
We have
\begin{equation}
\mathrm{Div}(f^*\gamma) = Y_1' - l Y_2' \;,
\end{equation}
By \cite[Th{\'e}or{\`e}me 7.31]{v},
we have
\begin{align}
\begin{split}
& H^\bullet_\mathrm{dR}(X') = H^\bullet_\mathrm{dR}(X) \oplus H_\mathrm{dR}^{\bullet-2}(Z) \;,\\
& H^\bullet_\mathrm{dR}(Y_1') = H^\bullet_\mathrm{dR}(Y_1) \oplus H^\bullet_\mathrm{dR}(Z) \oplus H_\mathrm{dR}^{\bullet-2}(Z) \;,\hspace{5mm}
H^\bullet_\mathrm{dR}(Y_2') = H^\bullet_\mathrm{dR}(Y_2) \;.
\end{split}
\end{align}
As a consequence,
we have
\begin{equation}
\label{eqchi-rem-tau-aux-bl}
\chi(X') = \chi(X) + \chi(Z) \;,\hspace{5mm}
\chi(Y_1') = \chi(Y_1) + 2 \chi(Z) \;,\hspace{5mm}
\chi(Y_2') = \chi(Y_2) \;.
\end{equation}
By \eqref{eq-def-w-gamma} and \eqref{eqchi-rem-tau-aux-bl},
we have
\begin{equation}
\label{eq1-rem-tau-aux-bl}
w(X',f^*\gamma) = w(X,\gamma) \;.
\end{equation}
By \eqref{eq-z} and \eqref{eq1-rem-tau-aux-bl},
for $z\in\C^*$,
we have
\begin{equation}
\tau(X',zf^*\gamma) - \tau(X,z\gamma) =
\tau(X',f^*\gamma) - \tau(X,\gamma) \;.
\end{equation}
\end{rem}

\begin{rem}
\label{rem-tau-bl}
Let $f:X' \rightarrow X$ be as in Remark \ref{rem-tau-aux-bl}.
We further assume that $\gamma\in H^0(X,K_X)$.
Let $Y\subseteq X$ (resp. $Y'\subseteq X'$)
be the zero locus of $\gamma$ (resp. $f^*\gamma$).
By \eqref{eq-def-tau}, \eqref{eq-def-tau-aux}, \eqref{eq-w-w1} and \eqref{eq1-rem-tau-aux-bl},
we have
\begin{equation}
\tau(X',Y') - \tau(X,Y) =
\tau(X',f^*\gamma) - \tau(X,\gamma) \;.
\end{equation}
\end{rem}

\section{Example}
\label{sect-ex}

\subsection{Proof of Theorem \ref{intro-thm-ex}}

Let $(X,Y)$ be a $(-2)$-Calabi-Yau pair
such that $X$ is a del Pezzo surface,
i.e., $\dim X=2$ and $K_X^{-1}$ is ample.

Let $\gamma \in H^0(X,K_X^{-2})$ such that $\mathrm{Div}(\gamma) = Y$.
Let $\nabla^{K_X^{-2}}$ be a connection on $K_X^{-2}$.
Let $N_Y$ be the normal bundle of $Y\subseteq X$.
Set
\begin{equation}
\label{eq-gammaprim-2}
\gamma' = \nabla^{K_X^{-2}} \gamma\big|_Y
\in H^0(Y,N_Y^{-1} \otimes K_X^{-2})
= H^0(Y,K_Y^{-1} \otimes K_X^{-1}) \;,
\end{equation}
which is independent of $\nabla^{K_X^{-2}}$.
Let
\begin{equation}
\big(\gamma'\big)^{-1} \in H^0(Y,K_Y \otimes K_X)
\end{equation}
be the inverse of $\gamma'$.

Let $j: Y \rightarrow X$ be the canonical embedding.
We have a short exact sequence of analytic coherent sheaves on $X$,
\begin{equation}
\label{eq-resol-Y}
0 \rightarrow
\mathscr{O}_X(K_X^2) \xrightarrow{\gamma}
\mathscr{O}_X \rightarrow
j_*\mathscr{O}_Y \rightarrow 0 \;.
\end{equation}
Since $K_X^{-1}$ is ample,
by the Kodaira vanishing theorem and Serre duality,
we have
\begin{equation}
\label{eq-vanishing}
H^0(X,K_X^2) = H^1(X,K_X^2) = 0 \;,\hspace{5mm}
H^{0,1}(X) = H^{0,2}(X) = 0 \;.
\end{equation}
Taking the long exact sequence induced by \eqref{eq-resol-Y}
and applying \eqref{eq-vanishing},
we get
\begin{equation}
\label{eq-iso-XY}
H^{0,0}(X) \xrightarrow{\sim} H^{0,0}(Y) \;,\hspace{5mm}
H^{0,1}(Y) \xrightarrow{\sim} H^2(X,K_X^2) \;.
\end{equation}
Let $1_X\in H^{0,0}(X)$ (resp. $1_Y\in H^{0,0}(Y)$) be the constant function $1$ on $X$ (resp. $Y$).
The first isomorphism in \eqref{eq-iso-XY}
could be explicitly calculated as follows,
\begin{align}
\label{eq2-iso-XY}
\begin{split}
H^{0,0}(X) & \rightarrow H^{0,0}(Y) \\
1_X & \mapsto 1_Y \;.
\end{split}
\end{align}
The second isomorphism in \eqref{eq-iso-XY}
is the Serre dual of the following map,
\begin{align}
\label{eq1-iso-XY}
\begin{split}
H^0(X,K_X^{-1}) & \rightarrow H^{1,0}(Y) \\
\phi & \mapsto \big(\gamma'\big)^{-1}\phi\big|_Y \;.
\end{split}
\end{align}
Set
\begin{equation}
\label{eq-def-lambda1-lambda2}
\xi_1 = \det H^{0,\bullet}(X) = H^{0,0}(X) \;,\hspace{5mm}
\xi_2 = \det H^\bullet(X,K_X^2) = \det H^2(X,K_X^2) \;.
\end{equation}
Let $\lambda_0(\cdot)$ be as in \eqref{eq-def-lambda-p} with $p=0$.
By \eqref{eq-iso-XY} and \eqref{eq-def-lambda1-lambda2},
we have
\begin{equation}
\label{eq-iso-lambda-XY}
\lambda_0(Y) = \xi_1 \otimes \xi_2^{-1} \;.
\end{equation}

Let $g$ be the genus of the curve $Y$.
By the second isomorphism in \eqref{eq-iso-XY},
we have
\begin{equation}
g = \dim H^2(X,K_X^2) \;.
\end{equation}
We fix a non-zero element
\begin{equation}
\label{eq-phiX}
\phi_X \in \xi_2 \;.
\end{equation}
Let $\phi_1,\cdots,\phi_g \in H^2(X,K_X^2)$ such that
$\phi_X = \phi_1\wedge\cdots\wedge\phi_g$.
Let $\varphi_1,\cdots,\varphi_g\in H^{0,1}(Y)$
be the pre-images of $\phi_1,\cdots,\phi_g$
via the second isomorphism in \eqref{eq-iso-XY}.
Let $(\cdot,\cdot\big)$
be the intersection form on $H^1_\mathrm{dR}(Y) = H^{1,0}(Y) \oplus H^{0,1}(Y)$.
Set
\begin{equation}
\label{eq-def-J}
J(\gamma,\phi_X) = \det
\Big((\varphi_j,\overline{\varphi}_k)_{1\leqslant j,k \leqslant g}\Big)
\in \C^* \;.
\end{equation}

Let $\omega$ be a K{\"a}hler form on $X$.

Let $g^{TX}$ (resp. $g^{TY}$) be the metric on $TX$ (resp. $TY$)
induced by $\omega$.
Let $\big|\cdot\big|$
be the norm on $K_X$ (resp. $K_Y^{-1}\otimes K_X^{-1}$)
induced by $\omega$.
We denote
\begin{align}
\label{eq-def-alphaXY-betaY}
\begin{split}
\alpha_X(\gamma,\omega) & = \frac{1}{12} \int_X
c_1^2\big(TX,g^{TX}\big)
\log \big|\gamma\big|^2 \;,\\
\alpha_Y(\gamma,\omega) & = \frac{1}{12} \int_Y
c_1\big(TX,g^{TX}\big)
\log \big|\gamma'\big|^2 \;,\\
\beta_Y(\gamma,\omega) & = \frac{1}{12} \int_Y
\log \big|\gamma'\big|^2
\frac{\overline{\partial}\partial}{2\pi i}
\log \big|\gamma'\big|^2 \;.
\end{split}
\end{align}
Let $a_X(\gamma,\omega)$, $a_Y(\gamma,\omega)$, $b_Y(\omega)$
and $\tau_\mathrm{BCOV}(X,\omega)$
be as in \textsection \ref{subsect-construction}.
Let \hbox{$\big\lVert \cdot \big\rVert_{\xi_1,\omega}$}
(resp. \hbox{$\big\lVert \cdot \big\rVert_{\xi_2,\omega}$})
be the Quillen metric on $\xi_1$ (resp. $\xi_2$) associated with $\omega$.
Recall that $w_\cdot(\cdot)$ was defined by \eqref{eq-def-w}.

\begin{prop}
\label{prop-resol}
The following identity holds,
\begin{align}
\label{eq-prop-resol}
\begin{split}
\tau(X,Y)
& = \tau_\mathrm{BCOV}(X,\omega)
- \log \big\lVert 1_X \big\rVert^2_{\xi_1,\omega}
+ \log \big\lVert \phi_X \big\rVert^2_{\xi_2,\omega}
- \log \big| J(\gamma,\phi_X) \big| \\
& \hspace{5mm} + 3\alpha_X(\gamma,\omega) + \frac{3}{2}a_X(\gamma,\omega)
+ \frac{9}{2}\alpha_Y(\gamma,\omega) - \frac{3}{2} b_Y(\omega) + \frac{3}{2} \beta_Y(\gamma,\omega) \\
& \hspace{5mm} + w_{-2}(X) \log \int_X \big|\gamma\overline{\gamma}\big|^{-1/2}
+ \mathrm{constant} \;.
\end{split}
\end{align}
Here 'constant' means a real number determined by the topology of $X$.
\end{prop}
\begin{proof}
The proof consists of several steps.

\noindent\textbf{Step 1.}
We show that
the right hand side of \eqref{eq-prop-resol}
is independent of $\omega$.

Let $\big(\omega_s\big)_{s\in{\C}P^1}$ be a smooth family of K{\"a}hler forms on $X$.
Let $\tau_\mathrm{BCOV}(X,\omega)$ be
the function $s\mapsto \tau_\mathrm{BCOV}(X,\omega_s)$ on ${\C}P^1$.
In the same way,
all the terms on right hand side of \eqref{eq-prop-resol}
are viewed as functions on ${\C}P^1$.
We need to show that
the right hand side of \eqref{eq-prop-resol} is a constant function on ${\C}P^1$.

We view $TX$ as a holomorphic vector bundle over $X\times{\C}P^1$.
Let $g^{TX}$ be the metric on $TX$ induced by $\big(\omega_s\big)_{s\in{\C}P^1}$,
i.e., $g^{TX}\big|_{X\times\{s\}}$ is induced by $\omega_s$.
We view $K_X^2$ as a holomorphic line bundle over $X\times{\C}P^1$.
Let $g^{K_X^2}$ be the metric on $K_X^2$ induced by $g^{TX}$.
By \cite[Theorem 0.1]{bgs1} and the Poincar{\'e}-Lelong formula,
we have
\begin{align}
\label{eq11-pf-prop-resol}
\begin{split}
& - \frac{\overline{\partial}\partial}{2\pi i}
\log \big\lVert 1_X \big\rVert^2_{\xi_1,\omega}
+ \frac{\overline{\partial}\partial}{2\pi i}
\log \big\lVert \phi_X \big\rVert^2_{\xi_2,\omega} \\
& = \bigg\{ \int_X \mathrm{Td}\big(TX,g^{TX}\big)
\Big(1 - \mathrm{ch}\big(K_X^2,g^{K_X^2}\big)\Big) \bigg\}^{(1,1)} \\
& = \frac{1}{2}\int_X c_1^3\big(TX,g^{TX}\big)
+ \frac{1}{6}\int_X c_1\big(TX,g^{TX}\big)c_2\big(TX,g^{TX}\big) \;.
\end{split}
\end{align}

By the Poincar{\'e}-Lelong formula
and the first identity in \eqref{eq-def-alphaXY-betaY},
we have
\begin{equation}
\label{eq12-pf-prop-resol}
\frac{\overline{\partial}\partial}{2\pi i} \alpha_X(\gamma,\omega)
= \frac{1}{12} \int_Y c_1^2\big(TX,g^{TX}\big)
- \frac{1}{6} \int_X c_1^3\big(TX,g^{TX}\big) \;.
\end{equation}

We view $TY$ (resp. $N_Y$) as a holomorphic line bundle over $Y\times{\C}P^1$.
Let $g^{TY}$ (resp. $g^{N_Y}$) be
the metric on $TY$ (resp. $N_Y$) induced by $g^{TX}$.
By the Poincar{\'e}-Lelong formula,
we have
\begin{align}
\label{eq13-pf-prop-resol}
\begin{split}
\frac{\overline{\partial}\partial}{2\pi i} \log \big|\gamma'\big|^2
& = c_1(N_Y,g^{N_Y}) - 2c_1(TX,g^{TX})\big|_{Y\times{\C}P^1} \\
& = - c_1(TY,g^{TY}) - c_1(TX,g^{TX})\big|_{Y\times{\C}P^1}
\in \Omega^{1,1}(Y\times{\C}P^1) \;.
\end{split}
\end{align}

By the last two identities in \eqref{eq-def-alphaXY-betaY}
and \eqref{eq13-pf-prop-resol},
we have
\begin{align}
\label{eq14-pf-prop-resol}
\begin{split}
\frac{\overline{\partial}\partial}{2\pi i} \alpha_Y(\gamma,\omega) &
= - \frac{1}{12} \int_Y c_1(TX,g^{TX}) \Big(c_1(TY,g^{TY}) + c_1(TX,g^{TX})\Big) \;, \\
\frac{\overline{\partial}\partial}{2\pi i} \beta_Y(\gamma,\omega) &
= \frac{1}{12} \int_Y \Big(c_1(TY,g^{TY}) + c_1(TX,g^{TX})\Big)^2 \;.
\end{split}
\end{align}

By \eqref{eq-anomaly-XY}-\eqref{eq-bc2}
and \eqref{eq11-pf-prop-resol}-\eqref{eq14-pf-prop-resol},
the right hand side of \eqref{eq-prop-resol} is a harmonic function on ${\C}P^1$.
Hence it is constant.

Let $\lambda_p(\cdot)$ be as in \eqref{eq-def-lambda-p}.
Let $\big\lVert\cdot\big\rVert_{\lambda_p(Y),\omega}$
be the Quillen metric on $\lambda_p(Y)$ associated with $\omega\big|_Y$.

By \eqref{eq-def-lambda1-lambda2}, \eqref{eq-iso-lambda-XY} and \eqref{eq-phiX},
we have $\phi_X^{-1} \otimes 1_X \in \lambda_0(Y)$.

\noindent\textbf{Step 2.}
We show that
\begin{align}
\label{eq2-pf-prop-resol}
\begin{split}
\tau_\mathrm{BCOV}(Y,\omega|_Y)
& = - \log \Big\lVert \phi_X^{-1} \otimes 1_X \Big\rVert^2_{\lambda_0(Y),\omega}
- \log \big| J(\gamma,\phi_X) \big|
+ \mathrm{constant} \;.
\end{split}
\end{align}

Let $1_Y\in H^{0,0}(Y)$ be the constant function $1$ on $Y$.
Let $1_Y^*\in H^{1,1}(Y)$ be its dual.
Let $\alpha_1,\cdots,\alpha_{2g}\in H^1_\mathrm{dR}(Y)$
be a basis of the lattice $H^1_\mathrm{Sing}(Y,\Z) \hookrightarrow H^1_\mathrm{dR}(Y)$.
Set
\begin{equation}
\alpha_Y = \alpha_1\wedge\cdots\wedge\alpha_{2g}\in \det H^1_\mathrm{dR}(Y) \;.
\end{equation}

In the sequel,
for $\xi$ a product of several elements in
\begin{equation}
\Big\{ \lambda_p(Y), \overline{\lambda_p(Y)}, \lambda_p(Y)^{-1}, \overline{\lambda_p(Y)}^{-1}
\;:\; p=0,1 \Big\} \;,
\end{equation}
we denote by $\big\lVert\cdot\big\rVert_{\xi,\omega}$
the metric on $\xi$ induced by $\big\lVert\cdot\big\rVert_{\lambda_p(Y),\omega}$.
Moreover,
for ease of notation,
we denote $(\lambda_p\otimes\lambda_q)(Y) = \lambda_p(Y)\otimes\lambda_q(Y)$.
By \eqref{eq-def-bcov-torsion},
we have
\begin{align}
\label{eq21-pf-prop-resol}
\begin{split}
& \tau_\mathrm{BCOV}(Y,\omega|_Y) \\
& = \log \Big\lVert 1_Y^{*,2} \otimes \alpha_Y^{-1} \Big\rVert_{\big(\lambda_1^{-1}\otimes\overline{\lambda}_1^{-1}\big)(Y),\omega} \\
& = \log \Big\lVert 1_Y^{*,2} \otimes \alpha_Y^{-2} \otimes 1_Y^2
\Big\rVert_{\big(\lambda_0\otimes\overline{\lambda}_0
\otimes\lambda_1^{-1}\otimes\overline{\lambda}_1^{-1}\big)(Y),\omega}
- \log \Big\lVert \alpha_Y^{-1} \otimes 1_Y^2
\Big\rVert_{\big(\lambda_0\otimes\overline{\lambda}_0\big)(Y),\omega} \\
& = \log \Big\lVert 1_Y^* \otimes \alpha_Y^{-1} \otimes 1_Y
\Big\rVert_{\big(\lambda_0\otimes\lambda_1^{-1}\big)(Y),\omega}^2
- \log \Big\lVert \alpha_Y^{-1} \otimes 1_Y^2
\Big\rVert_{\big(\lambda_0\otimes\overline{\lambda}_0\big)(Y),\omega} \;.
\end{split}
\end{align}

We identify $H^{0,0}(Y)$ with $H^{0,0}(X)$ via the first isomorphism in \eqref{eq-iso-XY}.
We identify $H^{0,1}(Y)$ with $H^2(X,K_X^2)$ via the second isomorphism in \eqref{eq-iso-XY}.
Recall that $J(\gamma,\phi_X)$ was defined by \eqref{eq-def-J}.
We have
\begin{equation}
\label{eqJ-pf-prop-resol}
1_Y = 1_X \;,\hspace{5mm}
\alpha_Y  = \pm \big(J(\gamma,\phi_X)\big)^{-1} \phi_X \otimes \overline{\phi_X} \;.
\end{equation}
By \eqref{eqJ-pf-prop-resol},
we have
\begin{align}
\label{eq22-pf-prop-resol}
\begin{split}
& \log \Big\lVert \alpha_Y^{-1} \otimes 1_Y^2
\Big\rVert_{\big(\lambda_0\otimes\overline{\lambda}_0\big)(Y),\omega} \\
& = \log \Big\lVert \big(\phi_X \otimes \overline{\phi_X}\big)^{-1} \otimes 1_Y^2
\Big\rVert_{\big(\lambda_0\otimes\overline{\lambda}_0\big)(Y),\omega}
+ \log \big| J(\gamma,\phi_X) \big| \\
& = \log \Big\lVert \phi_X^{-1} \otimes 1_Y
\Big\rVert_{\lambda_0(Y),\omega}^2
+ \log \big| J(\gamma,\phi_X) \big| \;.
\end{split}
\end{align}

By \cite[page 1304]{b04},
the real number
\begin{equation}
\log \Big\lVert 1_Y^* \otimes \alpha_Y^{-1} \otimes 1_Y
\Big\rVert_{\big(\lambda_0\otimes\lambda_1^{-1}\big)(Y),\omega}^2
\end{equation}
is determined by the genus of $Y$.
Then \eqref{eq2-pf-prop-resol} follows from \eqref{eq21-pf-prop-resol} and \eqref{eq22-pf-prop-resol}.

\noindent\textbf{Step 3.}
We conclude.

By Step 1, it is sufficient to prove \eqref{eq-prop-resol}
with a K{\"a}hler form $\omega$ satisfying
$\big|\gamma'\big| = 1$.
This assumption implies
\begin{equation}
\label{eq30-pf-prop-resol}
a_Y(\gamma,\omega) = 0 \;,\hspace{5mm}
\alpha_Y(\gamma,\omega) = 0 \;,\hspace{5mm}
\beta_Y(\gamma,\omega) =0 \;.
\end{equation}

Applying Theorem \ref{thm-quillen-immersion} with $L=K_X^2$,
we obtain
\begin{align}
\label{eq3-pf-prop-resol}
\begin{split}
\log \Big\lVert \phi_X^{-1} \otimes 1_X \Big\rVert^2_{\lambda_0(Y),\omega}
& = \log \big\lVert 1_X \big\rVert^2_{\xi_1,\omega}
- \log \big\lVert \phi_X \big\rVert^2_{\xi_2,\omega} \\
& \hspace{5mm} - 3\alpha_X(\gamma,\omega) - a_X(\gamma,\omega) + b_Y(\omega)
+ \mathrm{constant} \;.
\end{split}
\end{align}
By \eqref{eq-def-tau-omega}, \eqref{eq-def-tau}
and the first identity in \eqref{eq30-pf-prop-resol},
we have
\begin{align}
\label{eq31-pf-prop-resol}
\begin{split}
\tau(X,\gamma,\omega) =
& \; \tau_\mathrm{BCOV}(X,\omega) + \tau_\mathrm{BCOV}(Y,\omega|_Y) \\
& \;  + \frac{1}{2} a_X(\gamma,\omega)
- \frac{1}{2}b_Y(\omega)
+ w_{-2}(X) \log \int_X \big|\gamma\overline{\gamma}\big|^{-1/2} \;.
\end{split}
\end{align}
From \eqref{eq2-pf-prop-resol} and \eqref{eq30-pf-prop-resol}-\eqref{eq31-pf-prop-resol},
we obtain \eqref{eq-prop-resol} with $\omega$ satisfying $\big|\gamma'\big| = 1$ .
This completes the proof.
\end{proof}

We denote
\begin{equation}
D = \big\{t\in\C\;:\;|t|<1\big\} \;,\hspace{5mm}
D^* = \big\{t\in\C\;:\;0<|t|<1\big\} \;.
\end{equation}
Let
\begin{equation}
\Big(\gamma_t\in H^0(X,K_X^{-2})\backslash\{0\}\Big)_{t\in D}
\end{equation}
be a holomorphic family.
Let $Y_t\subseteq X$ be the zero locus of $\gamma_t$.
Let $l\in \N$.
We assume that
\begin{itemize}
\item[-] the family $\big(Y_t\big)_{t\in D^*}$ is smooth;
\item[-] the family $\big(Y_t\big)_{t\in D}$ has exactly $l$ ordinary double points
(cf. \cite[Proposition 3.2 (ii)]{bb}) at $t=0$.
\end{itemize}

\begin{prop}
\label{prop-asymp-tau}
As $t\rightarrow 0$,
we have
\begin{equation}
\label{eq-prop-asymp-tau}
\tau(X,Y_t) = \frac{l}{8} \log |t|^2
+ \mathscr{O}\big(\log(-\log|t|)\big) \;.
\end{equation}
\end{prop}
\begin{proof}
Let $x_1,\cdots,x_l\in X$ be the singular points of $Y_0\subseteq X$.
Let $\omega\in\Omega^{1,1}(X)$ be a K{\"a}hler form
such that the curvature of the metric on $TX$ induced by $\omega$
vanishes near $x_1,\cdots,x_l$.

By Proposition \ref{prop-resol},
we have
\begin{align}
\label{eqtau-pf-prop-asymp-tau}
\begin{split}
\tau(X,Y_t)
& = \tau_\mathrm{BCOV}(X,\omega)
- \log \big\lVert 1_X \big\rVert^2_{\xi_1,\omega}
+ \log \big\lVert \phi_X \big\rVert^2_{\xi_2,\omega}
- \log \big| J(\gamma_t,\phi_X) \big| \\
& \hspace{5mm} + 3\alpha_X(\gamma_t,\omega) + \frac{3}{2}a_X(\gamma_t,\omega)
+ \frac{9}{2}\alpha_{Y_t}(\gamma_t,\omega) - \frac{3}{2}b_{Y_t}(\omega)
+ \frac{3}{2}\beta_{Y_t}(\gamma_t,\omega) \\
& \hspace{5mm} + w_{-2}(X) \log \int_X \big|\gamma_t\overline{\gamma_t}\big|^{-1/2}
+ \mathrm{constant} \;.
\end{split}
\end{align}

Since the K{\"a}hler form $\omega$ is  independent of $t$,
as $t\rightarrow 0$,
we obviously have
\begin{equation}
\label{eqx-pf-prop-asymp-tau}
\tau_\mathrm{BCOV}(X,\omega) = \mathscr{O}\big(1\big) \;,\hspace{5mm}
\log \big\lVert 1_X \big\rVert^2_{\xi_1,\omega} = \mathscr{O}\big(1\big)\;,\hspace{5mm}
\log \big\lVert \phi_X \big\rVert^2_{\xi_2,\omega} = \mathscr{O}\big(1\big) \;.
\end{equation}

Since $c_1(TX,g^{TX})$ and $c_2(TX,g^{TX})$ vanish near $x_1,\cdots,x_l$,
as $t\rightarrow 0$,
we have
\begin{equation}
\alpha_X(\gamma_t,\omega) = \mathscr{O}\big(1\big) \;,\hspace{5mm}
a_X(\gamma_t,\omega) = \mathscr{O}\big(1\big) \;,\hspace{5mm}
\alpha_{Y_t}(\gamma_t,\omega) = \mathscr{O}\big(1\big) \;.
\end{equation}

Proceeding in the same way as
in Step 3 of the proof of \cite[Theorem 5.1]{y07},
as $t\rightarrow 0$,
we have
\begin{equation}
b_{Y_t}(\omega) = \mathscr{O}\big(1\big) \;.
\end{equation}

Proceeding in the same way as in the proof of \cite[Theorem 4.1]{y98},
as $t\rightarrow 0$,
we have
\begin{equation}
\beta_{Y_t}(\gamma_t,\omega)
= \frac{l}{12} \log |t|^2 + \mathscr{O}\big(1\big) \;.
\end{equation}

By a direct calculation,
as $t\rightarrow 0$,
we have
\begin{equation}
\label{eqint-pf-prop-asymp-tau}
\log \int_X \big|\gamma_t\overline{\gamma_t}\big|^{-1/2} = \mathscr{O}\big(\log(-\log|t|)\big) \;.
\end{equation}

Let
\begin{equation}
\label{eqiso-pf-prop-asymp-tau}
H^{0,1}(Y_t) \xrightarrow{\sim} H^2(X,K_X^2)
\end{equation}
be the second isomorphism in \eqref{eq-iso-XY} with $Y$ replaced by $Y_t$.
Recall that
\begin{equation}
\label{eqphi-pf-prop-asymp-tau}
\big(\phi_k\in H^2(X,K_X^2)\big)_{1\leqslant k\leqslant g}
\end{equation}
is a basis satisfying $\phi_X = \phi_1\wedge\cdots\wedge\phi_g$.
Let
\begin{equation}
\label{eqvarphi-pf-prop-asymp-tau}
\big(\varphi_k(t)\in H^{0,1}(Y_t)\big)_{1\leqslant k\leqslant g}
\end{equation}
be the pre-image of \eqref{eqphi-pf-prop-asymp-tau} via the isomorphism \eqref{eqiso-pf-prop-asymp-tau}.
Let
\begin{equation}
\label{eqisodual-pf-prop-asymp-tau}
H^0(X,K_X^{-1}) \xrightarrow{\sim} H^{1,0}(Y_t)
\end{equation}
be the isomorphism in \eqref{eq1-iso-XY} with $Y$ replaced by $Y_t$,
which is the Serre dual of \eqref{eqiso-pf-prop-asymp-tau}.
Let $\big(\cdot,\cdot\big)_X$ be the Serre pairing between $H^0(X,K_X^{-1})$ and $H^2(X,K_X^2)$.
Let
\begin{equation}
\label{eqeps-pf-prop-asymp-tau}
\big(\epsilon_k\in H^0(X,K_X^{-1})\big)_{1\leqslant k\leqslant g}
\end{equation}
be the dual basis of \eqref{eqphi-pf-prop-asymp-tau} with respect to $\big(\cdot,\cdot\big)_X$.
Let $\big(\cdot,\cdot\big)_t$ is the Serre pairing between $H^{1,0}(Y_t)$ and $H^{0,1}(Y_t)$.
Let
\begin{equation}
\label{eqvareps-pf-prop-asymp-tau}
\big(\varepsilon_k(t)\in H^{1,0}(Y_t)\big)_{1\leqslant k\leqslant g}
\end{equation}
be the dual basis of \eqref{eqvarphi-pf-prop-asymp-tau} with respect to $\big(\cdot,\cdot\big)_t$.
Then \eqref{eqvareps-pf-prop-asymp-tau} is the image of \eqref{eqeps-pf-prop-asymp-tau}
via the isomorphism \eqref{eqisodual-pf-prop-asymp-tau}.
Moreover,
by \eqref{eq1-iso-XY},
we have
\begin{equation}
\label{eqeps2vareps-pf-prop-asymp-tau}
\varepsilon_k(t) = \big(\gamma_t'\big)^{-1}\epsilon_k\big|_{Y_t} \;.
\end{equation}
By \eqref{eq-def-J} and the fact that
\eqref{eqvarphi-pf-prop-asymp-tau} and \eqref{eqvareps-pf-prop-asymp-tau} are dual to each other,
we have
\begin{align}
\label{eqJa-pf-prop-asymp-tau}
\begin{split}
\log \big| J(\gamma_t,\phi_X) \big|
& = \log \bigg| \det \Big(\big(\overline{\varphi_j(t)},\varphi_k(t)\big)_{t,1\leqslant j,k \leqslant g}\Big) \bigg| \\
& = - \log \bigg| \det \Big(\big(\varepsilon_j(t),\overline{\varepsilon_k(t)}\big)_{t,1\leqslant j,k \leqslant g}\Big) \bigg| \;.
\end{split}
\end{align}
We remark that
the $L^2$-metric on $H^{1,0}(Y_t)$ with respect to any Hermitian metric on $TY$
is given by $\big(\cdot,\overline{\cdot}\big)_t$.
Applying \eqref{eqeps2vareps-pf-prop-asymp-tau}
and proceeding in the same way as in the proof of \cite[Proposition 7.1]{bb},
as $t\rightarrow 0$,
we have
\begin{equation}
\label{eqJb-pf-prop-asymp-tau}
\log \bigg| \det \Big(\big(\varepsilon_j(t),\overline{\varepsilon_k(t)}\big)_{t,1\leqslant j,k \leqslant g}\Big) \bigg|
= \mathscr{O}\big(\log(-\log|t|)\big) \;.
\end{equation}
By \eqref{eqJa-pf-prop-asymp-tau} and \eqref{eqJb-pf-prop-asymp-tau},
we have
\begin{equation}
\label{eqJ-pf-prop-asymp-tau}
\log \big| J(\gamma_t,\phi_X) \big| = \mathscr{O}\big(\log(-\log|t|)\big) \;.
\end{equation}

From \eqref{eqtau-pf-prop-asymp-tau}-\eqref{eqint-pf-prop-asymp-tau}
and \eqref{eqJ-pf-prop-asymp-tau},
we obtain \eqref{eq-prop-asymp-tau}.
This completes the proof.
\end{proof}

\begin{proof}[Proof of Theorem \ref{intro-thm-ex}]
For readers' convenience,
we restate the setting in Theorem \ref{intro-thm-ex}.
Let $X$ be a del Pezzo surface such that $K_X^{-2}$ is very ample.
Let $\gamma\in H^0(X,K_X^{-2})$.
Let $Y$ be the zero locus of $\gamma$.
We assume that $Y$ is smooth and reduced.
Let $f: X'\rightarrow X$ be the ramified double covering
whose branch locus is $Y$.
Let $\iota$ be the involution on $X'$ commuting with $f$.
Let $\tau(X',\iota)$ be the logarithm of
Yoshikawa's equivariant BCOV invariant \cite[Definition 5.1]{y04} of $(X',\iota)$.
We need to show that
$\tau(X,Y)+\tau(X',\iota)$ is independent of $Y$.

Let
\begin{equation}
\Big(Y_s\in\big|K_X^{-2}\big|\Big)_{s\in S}
\end{equation}
be a generic curve in $\big|K_X^{-2}\big|$.
Since $K_X^{-2}$ is very ample,
there exists a finite subset $\Delta\subseteq S$ such that
\begin{itemize}
\item[-] the family $\big(Y_s\big)_{s\in S\backslash\Delta}$ is smooth;
\item[-] the family $\big(Y_s\big)_{s\in S}$ has
exactly one ordinary double points at each $s\in\Delta$.
\end{itemize}
For $s\in S\backslash\Delta$,
let $(X_s',\iota_s)$ be as in the last paragraph
with $Y$ replaced by $Y_s$.
Let $\tau(X,Y)$ (resp. $\tau(X',\iota)$)
be the function  $s\mapsto \tau(X,Y_s)$ (resp. $s\mapsto \tau(X_s',\iota_s)$)
on $S\backslash\Delta$.
It is sufficient to show that
$\tau(X,Y)+\tau(X',\iota)$ is constant on $S\backslash\Delta$.

\noindent\textbf{Step 1.}
We calculate $\overline{\partial}\partial\tau(X,Y)$.

Recall that $g$ is the genus of $Y$.
By the Hirzebruch-Riemann-Roch formula,
\eqref{eq-vanishing} and \eqref{eq-iso-XY},
we have
\begin{align}
\label{eq11-pf-prop-ex}
\begin{split}
g & = \int_X \mathrm{Td}\big(TX\big)\mathrm{ch}\big(K_X^2\big)
= \frac{13}{12}\int_X c_1^2\big(TX\big) + \frac{1}{12}\int_X c_2\big(TX\big) \;,\\
1 & = \int_X \mathrm{Td}\big(TX\big)
= \frac{1}{12}\int_X c_1^2\big(TX\big) + \frac{1}{12}\int_X c_2\big(TX\big) \;.
\end{split}
\end{align}
By \eqref{eq-def-w}
and \eqref{eq11-pf-prop-ex},
we have
\begin{equation}
\label{eq13-pf-prop-ex}
w_{-2}(X)
= \frac{1}{12}\int_X c_2\big(TX\big)
+ \frac{1}{12}\int_Y c_1\big(TY\big)
= \frac{13-g}{12} + \frac{2-2g}{12}
= \frac{5-g}{4} \;.
\end{equation}
Let $\omega_\mathrm{WP}\in\Omega^{1,1}(S\backslash\Delta)$
be the Weil-Petersson form (see \eqref{eq-intro-wp}) of $\big(X,Y_s\big)_{s\in S\backslash\Delta}$.
Let $H^\bullet(Y)$ be the variation of Hodge structure
associated with $\big(Y_s\big)_{s\in S\backslash\Delta}$.
Let $\omega_{H^\bullet(Y)}\in\Omega^{1,1}(S\backslash\Delta)$
be the Hodge form (see \eqref{eq-def-hodgeform}) of $H^\bullet(Y)$.
By Theorem \ref{intro-thm-curvature}
and \eqref{eq13-pf-prop-ex},
we have
\begin{equation}
\label{eq14a-pf-prop-ex}
\frac{\overline{\partial}\partial}{2\pi i} \tau(X,Y) =
\omega_{H^\bullet(Y)}
+ \frac{g-5}{4} \omega_{\mathrm{WP}} \;.
\end{equation}

Let $U\subseteq S$ be a small open subset.
Let $\Big(\gamma_s \in H^0(X,K_X^{-2})\Big)_{s\in U}$ be a holomorphic family
such that $\mathrm{Div}(\gamma_s) = Y_s$.
Let $\int_X \big|\gamma\overline{\gamma}\big|^{-1/2}$
be the function $s\mapsto\int_X \big|\gamma_s\overline{\gamma_s}\big|^{-1/2}$ on $U$.
By \eqref{eq-intro-wp},
we have
\begin{equation}
\label{eq14b-pf-prop-ex}
\omega_{\mathrm{WP}}\big|_U =
-\frac{\overline{\partial}\partial}{2\pi i} \log \int_X \big|\gamma\overline{\gamma}\big|^{-1/2} \;.
\end{equation}
For $s\in U$,
let $J(\gamma_s,\phi_X)$ be as in \eqref{eq-def-J} with $\gamma$ replaced by $\gamma_s$.
Let $J(\gamma,\phi_X)$ be the function $s\mapsto J(\gamma_s,\phi_X)$ on $U$.
By \eqref{eq-def-f}, \eqref{eq-def-hodgeform} and the second identity in \eqref{eqJ-pf-prop-resol},
we have
\begin{equation}
\label{eq14c-pf-prop-ex}
\omega_{H^\bullet(Y)}\big|_U =
\frac{\overline{\partial}\partial}{2\pi i} \log \big| J(\gamma,\phi_X) \big| \;.
\end{equation}
From \eqref{eq14a-pf-prop-ex}-\eqref{eq14c-pf-prop-ex},
we obtain
\begin{equation}
\label{eq14-pf-prop-ex}
\frac{\overline{\partial}\partial}{2\pi i} \tau(X,Y) \big|_U =
\frac{\overline{\partial}\partial}{2\pi i} \log \big| J(\gamma,\phi_X) \big|
+ \frac{5-g}{4} \frac{\overline{\partial}\partial}{2\pi i}
\log \int_X \big|\gamma\overline{\gamma}\big|^{-1/2} \;.
\end{equation}

\noindent\textbf{Step 2.}
We calculate $\overline{\partial}\partial\tau(X',\iota)$.

Let $M$ be the type of the $2$-elementary K3 surface $(X_s',\iota_s)$ (see \cite[Definition 1.4]{y04}).
Then $M$ is a lattice.
Following the notation in \cite{y04},
we denote by $r(M)$ be the rank of $M$.
By \cite[(1.9)]{y04} and \cite[Theorem 1.5]{y04},
we have
\begin{equation}
\label{eqrM-pf-prop-ex}
r(M) = 11 - g \;.
\end{equation}

For $s\in S\backslash\Delta$,
let $f_s : X_s' \rightarrow X$ be the ramified double covering
whose branch locus is $Y_s$.
Let $\Big(\eta_s\in H^0(X_s',K_{X_s'})\Big)_{s\in U}$ be a holomorphic family
such that
\begin{equation}
\label{eq15-pf-prop-ex}
\eta_s^2 = f_s^* \gamma_s^{-1} \;.
\end{equation}
Let $\int_{X'} \big|\eta\overline{\eta}\big|$
be the function
$s\mapsto \int_{X_s'} \big|\eta_s\overline{\eta}_s\big|$ on $U$.

Now we apply \cite[(5.15)]{y04} to the family $\big((X_s',\iota_s)\big)_{s\in U}$.
The terms on the right hand side of \cite[(5.15)]{y04}
are explained in \cite[(5.4)]{y04}
and the paragraph containing \cite[(5.13)]{y04}.
We obtain
\begin{equation}
\label{eq16-pf-prop-ex}
\frac{\overline{\partial}\partial}{2\pi i} \tau(X',\iota) \big|_U =
\frac{6-r(M)}{4} \frac{\overline{\partial}\partial}{2\pi i}
\log \int_{X'} \big|\eta\overline{\eta}\big|
- \frac{\overline{\partial}\partial}{2\pi i} \log \big| J(\gamma,\phi_X) \big| \;.
\end{equation}

\noindent\textbf{Step 3.}
We conclude.

By \eqref{eq15-pf-prop-ex},
we have
\begin{equation}
\label{eq17-pf-prop-ex}
2 \int_X \big|\gamma\overline{\gamma}\big|^{-1/2} =
\int_{X'} \big|\eta\overline{\eta}\big| \;.
\end{equation}
By \eqref{eq14-pf-prop-ex}, \eqref{eqrM-pf-prop-ex}, \eqref{eq16-pf-prop-ex} and \eqref{eq17-pf-prop-ex},
we have
\begin{equation}
\label{eq30-pf-prop-ex}
\overline{\partial}\partial
\big(\tau(X,Y) + \tau(X',\iota)\big) = 0 \;.
\end{equation}

Let $s_0\in\Delta$.
We identify a neighborhood of $s_0\in S$ with the unit disc $D$
such that $s_0$ is identified with $0\in D$.
Let $t\in D$ be the coordinate.
By Proposition \ref{prop-asymp-tau},
as $t\rightarrow 0$,
we have
\begin{equation}
\label{eq21-pf-prop-ex}
\tau(X,Y_t) = \frac{1}{8} \log |t|^2
+ \mathscr{O}\big(\log(-\log|t|)\big) \;.
\end{equation}
By \cite[Theorem 6.6]{y04},
as $t\rightarrow 0$,
we have
\begin{equation}
\label{eq32-pf-prop-ex}
\tau(X_t',\iota_t) = - \frac{1}{8} \log |t|^2
+ \mathscr{O}\big(\log(-\log|t|)\big) \;.
\end{equation}
By \eqref{eq30-pf-prop-ex}-\eqref{eq32-pf-prop-ex},
$\tau(X,Y)+\tau(X',\iota)$ is constant on $S\backslash\Delta$.
This completes the proof.
\end{proof}

\section{Behavior of $\tau(X,Y)$ under blow-up}
\label{sect-bl}

\subsection{Vanishing of curvature}

Let $S$ be a complex manifold.
Let $\big(X_s,Y_s\big)_{s\in S}$
be a holomorphic family of $1$-Calabi-Yau pairs.
We assume that $\big(X_s\big)_{s\in S}$, viewed as fibration over $S$, is locally K{\"a}hler.
Let $\big(Z_s \subseteq X_s \big)_{s\in S}$
be a holomorphic family of closed complex submanifolds of codimension $2$.
We assume that $Z_s \cap Y_s = \emptyset$ for any $s\in S$.
Let $f_s : X'_s \rightarrow X_s$ be the blow-up along $Z_s$.
Set $Y'_s = f^{-1}_s(Y_s \cup Z_s)$.
Then $\big(X'_s,Y'_s\big)_{s\in S}$ is
a holomorphic family of $1$-Calabi-Yau pairs.
Let $\tau(X,Y)$ (resp. $\tau(X',Y')$) be the function
$s\mapsto \tau(X_s,Y_s)$ (resp. $s\mapsto \tau(X'_s,Y'_s)$) on $S$.

\begin{prop}
\label{prop-bl}
The following identity holds,
\begin{equation}
\label{eq-prop-bl}
\overline{\partial}\partial \Big( \tau(X',Y') - \tau(X,Y) \Big) = 0 \;.
\end{equation}
\end{prop}
\begin{proof}
We consider the variations of Hodge structure
$H^\bullet(X)$, $H^\bullet(Y)$,
$H^\bullet(X')$, $H^\bullet(Y')$
and $H^\bullet(Z)$ over $S$.
Let $H(Z)[1]^\bullet$ be the first right shift of $H^\bullet(Z)$
(see \eqref{eq-def-shift}).
By \cite[Th{\'e}or{\`e}me 7.31]{v},
we have
\begin{equation}
\label{eq11-pf-prop-bl}
H^\bullet(X') = H^\bullet(X) \oplus H(Z)[1]^\bullet \;,\hspace{5mm}
H^\bullet(Y') = H^\bullet(Y) \oplus H^\bullet(Z) \oplus H(Z)[1]^\bullet \;.
\end{equation}
By Proposition \ref{prop-shift-hodgeform}
and \eqref{eq11-pf-prop-bl},
we have
\begin{equation}
\label{eq12-pf-prop-bl}
\omega_{H^\bullet(X')} = \omega_{H^\bullet(X)} + \omega_{H^\bullet(Z)} \;,\hspace{5mm}
\omega_{H^\bullet(Y')} = \omega_{H^\bullet(Y)} + 2 \omega_{H^\bullet(Z)} \;.
\end{equation}

By \eqref{eq-w-w1} and \eqref{eq1-rem-tau-aux-bl},
we have
\begin{equation}
\label{eq13-pf-prop-bl}
w_1(X') = w_1(X) \;.
\end{equation}

Let $\omega_\mathrm{WP}\in\Omega^{1,1}(S)$ (resp. $\omega_\mathrm{WP}'\in\Omega^{1,1}(S)$)
be the Weil-Petersson form (see \eqref{eq-intro-wp})
of $\big(X_s,Y_s\big)_{s\in S}$ (resp. $\big(X'_s,Y'_s\big)_{s\in S}$).
By \eqref{eq-intro-wp},
we have
\begin{equation}
\label{eq14-pf-prop-bl}
\omega_\mathrm{WP} = \omega_\mathrm{WP}' \;.
\end{equation}

By Theorem \ref{intro-thm-curvature},
we have
\begin{align}
\label{eq1-pf-prop-bl}
\begin{split}
\frac{\overline{\partial}\partial}{2\pi i}\tau(X,Y) & =
\omega_{H^\bullet(X)} - \frac{1}{2} \omega_{H^\bullet(Y)}
- w_1(X)\omega_\mathrm{WP} \;,\\
\frac{\overline{\partial}\partial}{2\pi i}\tau(X',Y') & =
\omega_{H^\bullet(X')} - \frac{1}{2} \omega_{H^\bullet(Y')}
- w_1(X')\omega_\mathrm{WP}' \;.
\end{split}
\end{align}
From \eqref{eq12-pf-prop-bl}-\eqref{eq1-pf-prop-bl},
we obtain \eqref{eq-prop-bl}.
This completes the proof.
\end{proof}

\subsection{Proof of Theorem \ref{intro-thm-bl}, \ref{intro-thm-bl-curve}}
\label{subsect-bl3}

In this subsection,
we will only prove Theorem \ref{intro-thm-bl-curve}.
Theorem \ref{intro-thm-bl} can be proved in exactly the same way.

Let $X$ be a complex manifold.
Let $Z\subseteq X$ be a closed complex submanifold.
Let $i_Z: Z \rightarrow X$ be the canonical embedding.
Let $N_Z$ be the normal bundle of $Z\subseteq X$.
Let $\mathcal{N}_Z$ be the total space of $N_Z$.
Let $j_Z: Z \rightarrow \mathcal{N}_Z$ be the embedding defined by the zero section of $N_Z$.
We say that $Z\subseteq X$ satisfies \textbf{Condition*}
if there exist open neighborhoods
\begin{equation}
\label{eq-def-U}
Z\subseteq U \subseteq X \;,\hspace{5mm}
j_Z(Z) \subseteq \mathcal{U} \subseteq \mathcal{N}_Z
\end{equation}
and a biholomorphic map
\begin{equation}
\label{eq-def-varphi}
\varphi: U \rightarrow \mathcal{U}
\end{equation}
such that the following diagram commutes
\begin{equation}
\label{eq2-def-varphi}
\xymatrix{
Z \ar[d]_{i_Z} \ar[dr]^{j_Z} & \\
U \ar[r]^{\hspace{-1.5mm}\varphi} & \mathcal{U} \;.
}
\end{equation}

\begin{prop}
\label{prop-a}
If $X$ is a threefold and $Z\subseteq X$ is a $(-1,-1)$-curve,
then $Z\subseteq X$ satisfies \textbf{Condition*}.
\end{prop}
\begin{proof}
Let $f: X' \rightarrow X$ be the blow-up along $Z$.
By \cite[page 363, Satz 7, Corollar]{gr},
$f^{-1}(Z)\subseteq X'$ satisfies \textbf{Condition*}.
Hence so does $Z\subseteq X$.
\end{proof}

\begin{proof}[Proof of Theorem \ref{intro-thm-bl-curve}]
For readers' convenience,
we restate the setting in Theorem \ref{intro-thm-bl-curve}.
Let $X$ be a compact K{\"a}hler manifold of dimension $3$.
Let $\gamma\in H^0(X,K_X)$ be a non-zero element.
Let $Y$ be the zero locus of $\gamma$.
We assume that $Y$ is smooth and reduced.
Let $Z\subseteq X$ be a $(-1,-1)$-curve such that $Z\cap Y=\emptyset$.
Let $f: X'\rightarrow X$ be the blow-up along $Z$.
Set $Y' = f^{-1}(Y\cup Z)$.
We need to show that
$\tau(X',Y')-\tau(X,Y)$ is independent of $X$, $Y$, $Z$.
The proof consists of several steps.

\noindent\textbf{Step 1.}
We construct $\nu$.

Let $N_Z$ be the normal bundle of $Z\subseteq X$.
Let $\mathcal{N}_Z$ be the total space of $N_Z$.
Set
\begin{equation}
\label{eq11-pf-intro-thm-bl-curve}
W = \mathbb{P}\big(N_Z\oplus\C\big)
= \mathcal{N}_Z \cup \mathbb{P}\big(N_Z\big) \;.
\end{equation}
Let
\begin{equation}
\label{eq12-pf-intro-thm-bl-curve}
\gamma_Z \in H^0\big(Z,K_X\big|_Z\big) = H^0\big(Z,K_Z \otimes (\Lambda^2 N_Z^*)\big)
\end{equation}
be the restriction of $\gamma$ to $Z$.
Let $\pi: \mathcal{N}_Z \rightarrow Z$ be the canonical projection.
Since $K_W\big|_{\mathcal{N}_Z} = \pi^*\big(K_Z \otimes (\Lambda^2 N_Z^*)\big)$,
we may view $\gamma_Z$ as an element in $\mathscr{M}(W,K_W)$.
Under the identification \eqref{eq11-pf-intro-thm-bl-curve},
we have
\begin{equation}
\label{eq13-pf-intro-thm-bl-curve}
\mathrm{Div}(\gamma_Z) = - 3 \mathbb{P}\big(N_Z\big) \;.
\end{equation}

Let $j_Z: Z \rightarrow \mathcal{N}_Z \subseteq W$
be the embedding defined by the zero section of $N_Z$.
Let $g: W' \rightarrow W$ be the blow-up along $j_Z(Z)\subseteq W$.
Let $\tau(\cdot,\cdot)$ be as in \eqref{eq-def-tau-aux}.
Set
\begin{equation}
\label{eq14-pf-intro-thm-bl-curve}
\nu = \tau(W',g^*\gamma_Z) - \tau(W,\gamma_Z) \;.
\end{equation}
Since $Z$ is a $(-1,-1)$-curve,
the complex manifolds $W$, $W'$,
the map $g: W' \rightarrow W$
and the line $[\gamma_Z]\in \mathbb{P}\big(\mathscr{M}(W,K_W)\big)$
are uniquely determined up to isomorphism.
Then,
by Remark \ref{rem-tau-aux-bl},
$\nu$ depends on nothing.
By Remark \ref{rem-tau-bl},
it remains to show that
\begin{equation}
\label{eq1-pf-intro-thm-bl-curve}
\tau(X',f^*\gamma) - \tau(X,\gamma) =
\tau(W',g^*\gamma_Z) - \tau(W,\gamma_Z) \;.
\end{equation}

\noindent\textbf{Step 2.}
We choose convenient K{\"a}hler forms.

By Proposition \ref{prop-a},
the curve $Z\subseteq X$ satisfies \textbf{Condition*}.
Let $U$, $\mathcal{U}$ and $\varphi$ be as in \eqref{eq-def-U}-\eqref{eq2-def-varphi}.
Let
$T\varphi\in H^0(Z,\mathrm{End}(N_Z))$
be the derivative of $\varphi$ on $Z$.
We may assume that
\begin{equation}
\label{eq2t-pf-intro-thm-bl-curve}
T\varphi = \mathrm{Id} \;.
\end{equation}
We also assume that
there exists a norm $\big|\cdot\big|$ on $N_Z$ such that
\begin{equation}
\label{eq2u-pf-intro-thm-bl-curve}
\mathcal{U} = \big\{ v \in\mathcal{N}_Z \;:\; \big|v\big| < 1 \big\} \;.
\end{equation}

We identity $\mathcal{U}$ with a subset of $W$
via the identification \eqref{eq11-pf-intro-thm-bl-curve}.
Set
\begin{equation}
U' = f^{-1}(U) \subseteq X' \;,\hspace{5mm}
\mathcal{U}' = g^{-1}(\mathcal{U}) \subseteq W' \;.
\end{equation}
Let $\varphi': U' \rightarrow \mathcal{U}'$ be the lift of $\varphi$,
i.e., the following diagram commutes,
\begin{equation}
\label{eq2diag-pf-intro-thm-bl-curve}
\xymatrix{
U' \ar[r]^{\varphi'} \ar[d]_f & \mathcal{U}' \ar[d]^g \\
U \ar[r]^{\varphi}           & \mathcal{U} \;.
}
\end{equation}

For $0<\varepsilon<1$,
set
\begin{equation}
\label{eq2ue-pf-intro-thm-bl-curve}
\mathcal{U}_\varepsilon =
\big\{ v \in\mathcal{N}_Z \;:\; \big|v\big| < \varepsilon \big\} \;.
\end{equation}
Set
\begin{equation}
U_\varepsilon = \varphi^{-1}(\mathcal{U}_\varepsilon) \;,\hspace{5mm}
U_\varepsilon' = f^{-1}(U_\varepsilon) \;,\hspace{5mm}
\mathcal{U}_\varepsilon' = g^{-1}(\mathcal{U}_\varepsilon) \;.
\end{equation}
For $\varepsilon$ small enough,
there exist a K{\"a}hler form $\omega_X$ on $X$ and a K{\"a}hler form $\omega_W$ on $W$ such that
\begin{equation}
\label{eq21-pf-intro-thm-bl-curve}
\omega_W\big|_{\mathcal{U}_\varepsilon}
= \varphi_* \big( \omega_X\big|_{U_\varepsilon} \big) \;.
\end{equation}
Let $\omega_{X'}$ be a K{\"a}hler form on $X'$ such that
\begin{equation}
\label{eq22-pf-intro-thm-bl-curve}
\omega_{X'} \big|_{X'\backslash U_\varepsilon'}
= f^* \big( \omega_X\big|_{X\backslash U_\varepsilon} \big) \;.
\end{equation}
Let $\omega_{W'}$ be the K{\"a}hler form on $W'$ defined by
\begin{equation}
\label{eq23-pf-intro-thm-bl-curve}
\omega_{W'} \big|_{W'\backslash \mathcal{U}_\varepsilon'}
= g^* \big( \omega_W\big|_{W\backslash\mathcal{U}_\varepsilon} \big) \;,\hspace{5mm}
\omega_{W'} \big|_{\mathcal{U}_\varepsilon'}
= \varphi_*' \big( \omega_{X'}\big|_{U_\varepsilon'} \big) \;.
\end{equation}

\noindent\textbf{Step 3.}
We prove that
\begin{equation}
\label{eq3-pf-intro-thm-bl-curve}
\tau_\mathrm{BCOV}(X',\omega_{X'}) - \tau_\mathrm{BCOV}(X,\omega_X) =
\tau_\mathrm{BCOV}(W',\omega_{W'}) - \tau_\mathrm{BCOV}(W,\omega_W) \;.
\end{equation}

We denote $D = f^{-1}(Z)\subseteq X'$.
We have
\begin{equation}
\label{eq31-pf-intro-thm-bl-curve}
D = {\C}P^1 \times Z \;.
\end{equation}
Although $Z \simeq {\C}P^1$,
we will intentionally use the notation in \eqref{eq31-pf-intro-thm-bl-curve}
in order to distinguish the two copies of ${\C}P^1$ in $D \simeq {\C}P^1\times{\C}P^1$.
Let $\mathrm{pr}_1: {\C}P^1 \times Z \rightarrow {\C}P^1$
and $\mathrm{pr}_2: {\C}P^1 \times Z \rightarrow Z$
be the canonical projections.
Let $L_1$ (resp. $L_2$) be the holomorphic line bundle over ${\C}P^1$ (resp. $Z$) of degree $-1$.
Let $N_D$ be the normal bundle of $D\subseteq X'$.
We have
\begin{equation}
\label{eq32a-pf-intro-thm-bl-curve}
N_D = \mathrm{pr}_1^*L_1 \otimes \mathrm{pr}_2^*L_2 \;.
\end{equation}
Since $D\subseteq X'$ satisfies \textbf{Condition*},
we have
\begin{equation}
\label{eq32-pf-intro-thm-bl-curve}
TX'\big|_D = T{\C}P^1 \oplus TZ \oplus N_D \;.
\end{equation}

Set
\begin{align}
\label{eq33-pf-intro-thm-bl-curve}
\begin{split}
& G^1 = T^*{\C}P^1
\subseteq T^*X'\big|_D \;,\\
& G^2 = \big( T^*{\C}P^1 \big) \otimes \big( T^*Z \oplus N_D^* \big)
\subseteq \Lambda^2\big(T^*X'\big)\big|_D \;,\\
& G^3 = \big( T^*{\C}P^1 \big) \otimes \big( T^*Z \big) \otimes N_D^*
= K_{X'}\big|_D \;.
\end{split}
\end{align}
By \eqref{eq32a-pf-intro-thm-bl-curve} and \eqref{eq33-pf-intro-thm-bl-curve},
we have
\begin{align}
\label{eq34-pf-intro-thm-bl-curve}
\begin{split}
& H^q(D,G^p) = 0 \;,\hspace{5mm}\text{for } (p,q) \neq (1,1),(2,2) \;;\\
& H^1(D,G^1) = H^{1,1}({\C}P^1) \otimes H^{0,0}(Z) \;; \hspace{5mm}
H^2(D,G^2) = H^{1,1}({\C}P^1) \otimes H^{1,1}(Z) \;.
\end{split}
\end{align}

Let $i_D: D\rightarrow X'$ be the canonical embedding.
For $p=1,2,3$,
we have the following short exact sequence of analytic coherent sheaves on $X'$,
\begin{equation}
\label{eq35-pf-intro-thm-bl-curve}
0 \rightarrow
f^* \Omega^p_X \rightarrow
\Omega^p_{X'} \rightarrow
i_{D,*} \mathscr{O}_D(G^p) \rightarrow 0 \;.
\end{equation}
Let $\lambda_p(\cdot)$ be as in \eqref{eq-def-lambda-p}.
Taking the determinant of the long exact sequence induced by \eqref{eq35-pf-intro-thm-bl-curve}
and using the fact that
\begin{equation}
H^{p,\bullet}(X) = H^\bullet(X,\Omega^p_X) = H^\bullet(X',f^*\Omega^p_X) \;,
\end{equation}
we get a canonical section
\begin{equation}
\label{eq36-pf-intro-thm-bl-curve}
\sigma_p^f \in \big(\lambda_p(X)\big)^{-1} \otimes \lambda_p(X') \otimes \big(\det H^\bullet(D,G^p)\big)^{-1} \;.
\end{equation}
Set
\begin{equation}
\label{eqG-pf-intro-thm-bl-curve}
\lambda(G^\bullet) =
H^{1,1}({\C}P^1) \otimes H^{0,0}(Z) \otimes \Big(H^{1,1}({\C}P^1) \otimes H^{1,1}(Z)\Big)^2 \;.
\end{equation}
Let $\lambda_\mathrm{dR}(\cdot)$ be as in \eqref{eq-def-lambda-dR}.
By \eqref{eq-def-lambda},
\eqref{eq-def-lambda-dR},
\eqref{eq34-pf-intro-thm-bl-curve},
\eqref{eq36-pf-intro-thm-bl-curve}
and \eqref{eqG-pf-intro-thm-bl-curve},
we have
\begin{equation}
\label{eq37-pf-intro-thm-bl-curve}
\bigotimes_{p=1}^3 \Big(\sigma_p^f \otimes \overline{\sigma_p^f}\Big)^{(-1)^pp} \in
\lambda_\mathrm{dR}^{-1}(X) \otimes \lambda_\mathrm{dR}(X')
\otimes \big(\lambda(G^\bullet)\big)^{-2} \;.
\end{equation}
Let
\begin{equation}
\label{eq38-pf-intro-thm-bl-curve}
\sigma_X \in \lambda_\mathrm{dR}(X) \;,\hspace{5mm}
\sigma_{X'} \in \lambda_\mathrm{dR}(X')
\end{equation}
be as in \eqref{eq-def-sigma},
which are well-defined up to $\pm 1$.
Note that $Z \simeq \C P^1$,
every cohomology group in \eqref{eqG-pf-intro-thm-bl-curve} admits a $\Z$-structure given by $H^\bullet_\mathrm{Sing}(\C P^1,\Z)$.
Similarly to \eqref{eq-def-sigma},
let
\begin{equation}
\tau \in \lambda(G^\bullet)
\end{equation}
be the product of generators of the integer points of the cohomology groups in \eqref{eqG-pf-intro-thm-bl-curve},
which is well-defined up to $\pm 1$.
By the first identity in \eqref{eq11-pf-prop-bl},
we have
\begin{equation}
\label{eq39-pf-intro-thm-bl-curve}
\bigotimes_{p=1}^3 \Big(\sigma_p^f \otimes \overline{\sigma_p^f}\Big)^{(-1)^pp} =
\pm \sigma_X^{-1}\otimes\sigma_{X'} \otimes \tau^{-2} \;.
\end{equation}

Proceeding in the same way as in the last paragraph
with $f: X' \rightarrow X$ replaced by $g: W' \rightarrow W$,
we get canonical sections
\begin{equation}
\sigma_p^g \in \big(\lambda_p(W)\big)^{-1} \otimes \lambda_p(W') \otimes \big(\det H^\bullet(D,G^p)\big)^{-1}
\end{equation}
and
\begin{equation}
\label{eq310-pf-intro-thm-bl-curve}
\bigotimes_{p=1}^3 \Big(\sigma_p^g \otimes \overline{\sigma_p^g}\Big)^{(-1)^pp} \in
\lambda_\mathrm{dR}^{-1}(W) \otimes \lambda_\mathrm{dR}(W')
\otimes \big(\lambda(G^\bullet)\big)^{-2} \;.
\end{equation}
Let
\begin{equation}
\label{eq311-pf-intro-thm-bl-curve}
\sigma_W \in \lambda_\mathrm{dR}(W) \;,\hspace{5mm}
\sigma_{W'} \in \lambda_\mathrm{dR}(W')
\end{equation}
be as in \eqref{eq-def-sigma}.
Similarly to \eqref{eq39-pf-intro-thm-bl-curve},
we have
\begin{equation}
\label{eq312-pf-intro-thm-bl-curve}
\bigotimes_{p=1}^3 \Big(\sigma_p^g \otimes \overline{\sigma_p^g}\Big)^{(-1)^pp} =
\pm \sigma_W^{-1}\otimes\sigma_{W'} \otimes \tau^{-2} \;.
\end{equation}

Let \hbox{$\big\lVert\cdot\big\rVert_{\lambda_p(X),\omega_X}$}
be the Quillen metric on $\lambda_p(X)$ associated with $\omega_X$.
Let \hbox{$\big\lVert\cdot\big\rVert_{\lambda_\mathrm{dR}(X),\omega_X}$}
be the metric on $\lambda_\mathrm{dR}(X)$
induced by \hbox{$\big\lVert\cdot\big\rVert_{\lambda_p(X),\omega_X}$}.
We define
\hbox{$\big\lVert\cdot\big\rVert_{\lambda_\mathrm{dR}(X'),\omega_{X'}}$},
\hbox{$\big\lVert\cdot\big\rVert_{\lambda_\mathrm{dR}(W),\omega_W}$} and
\hbox{$\big\lVert\cdot\big\rVert_{\lambda_\mathrm{dR}(W'),\omega_{W'}}$}
in the same way.
We equip $TD$ with the metric induced by $\omega_{X'}\big|_D$.
We equip $G^\bullet$ with the Hermitian metric induced by $\omega_{X'}$.
Let \hbox{$\big\lVert\cdot\big\rVert_{\lambda(G^\bullet),\omega_{X'}}$}
be the associated Quillen metric on $\lambda(G^\bullet)$.
We may equally view $D$ as a divisor in $W'$ and define
\hbox{$\big\lVert\cdot\big\rVert_{\lambda(G^\bullet),\omega_{W'}}$}
in the same way.

Let $\big\lVert\sigma^f_p\big\rVert$ (resp. $\big\lVert\sigma^g_p\big\rVert$)
be the norm of $\sigma^f_p$ (resp. $\sigma^g_p$) with respect to the Quillen metric.
By Corollary \ref{cor-quillen-bl},
\eqref{eq39-pf-intro-thm-bl-curve} and \eqref{eq312-pf-intro-thm-bl-curve},
we have
\begin{align}
\label{eq313-pf-intro-thm-bl-curve}
\begin{split}
& \log \big\lVert\sigma_{X'}\big\rVert_{\lambda_\mathrm{dR}(X'),\omega_{X'}}
- \log \big\lVert\sigma_X\big\rVert_{\lambda_\mathrm{dR}(X),\omega_X}
- \log \big\lVert\tau\big\rVert^2_{\lambda(G^\bullet),\omega_{X'}} \\
& = \sum_{p=1}^3 (-1)^pp \log \big\lVert\sigma^f_p\big\rVert^2
= \sum_{p=1}^3 (-1)^pp \log \big\lVert\sigma^g_p\big\rVert^2 \\
& = \log \big\lVert\sigma_{W'}\big\rVert_{\lambda_\mathrm{dR}(W'),\omega_{W'}}
- \log \big\lVert\sigma_W\big\rVert_{\lambda_\mathrm{dR}(W),\omega_W}
- \log \big\lVert\tau\big\rVert^2_{\lambda(G^\bullet),\omega_{W'}} \;.
\end{split}
\end{align}
By the second identity in \eqref{eq23-pf-intro-thm-bl-curve},
we have
\begin{equation}
\label{eq314-pf-intro-thm-bl-curve}
\big\lVert\cdot\big\rVert_{\lambda(G^\bullet),\omega_{X'}}
= \big\lVert\cdot\big\rVert_{\lambda(G^\bullet),\omega_{W'}} \;.
\end{equation}
From \eqref{eq-def-bcov-torsion},
\eqref{eq313-pf-intro-thm-bl-curve} and
\eqref{eq314-pf-intro-thm-bl-curve},
we obtain \eqref{eq3-pf-intro-thm-bl-curve}.

Recall that $a_\cdot(\cdot,\cdot)$ was defined by \eqref{eq-def-aX}.

\noindent\textbf{Step 4.}
We show that
\begin{equation}
\label{eq4-pf-intro-thm-bl-curve}
a_{X'}(f^*\gamma,\omega_{X'}) - a_X(\gamma,\omega_X) =
a_{W'}(g^*\gamma_Z,\omega_{W'}) - a_W(\gamma_Z,\omega_W) \;.
\end{equation}

Let $a_{U_\varepsilon}(\gamma,\omega_X)$
be the right hand side of \eqref{eq-def-aX}
with $\int_X$ replaced by $\int_{U_\varepsilon}$
and $\omega$ replaced by $\omega_X$.
We define
$a_{U_\varepsilon'}(f^*\gamma,\omega_{X'})$,
$a_{\mathcal{U}_\varepsilon}(\gamma_Z,\omega_W)$ and
$a_{\mathcal{U}_\varepsilon'}(g^*\gamma_Z,\omega_{W'})$
in the same way.
By \eqref{eq22-pf-intro-thm-bl-curve}
and the first identity in \eqref{eq23-pf-intro-thm-bl-curve},
we have
\begin{align}
\label{eq41-pf-intro-thm-bl-curve}
\begin{split}
a_{X'}(f^*\gamma,\omega_{X'})
- a_X(\gamma,\omega_X) & =
a_{U_\varepsilon'}(f^*\gamma,\omega_{X'})
- a_{U_\varepsilon}(\gamma,\omega_X) \;,\\
a_{W'}(g^*\gamma_Z,\omega_{W'})
- a_W(\gamma_Z,\omega_W) & =
a_{\mathcal{U}_\varepsilon'}(g^*\gamma_Z,\omega_{W'})
- a_{\mathcal{U}_\varepsilon}(\gamma_Z,\omega_W) \;.
\end{split}
\end{align}
By \eqref{eq2diag-pf-intro-thm-bl-curve},
\eqref{eq21-pf-intro-thm-bl-curve}
and the second identity in \eqref{eq23-pf-intro-thm-bl-curve},
we have
\begin{align}
\label{eq42-pf-intro-thm-bl-curve}
\begin{split}
a_{U_\varepsilon}(\gamma,\omega_X)
- a_{\mathcal{U}_\varepsilon}(\gamma_Z,\omega_W) & =
a_{\mathcal{U}_\varepsilon}(\varphi_*\gamma,\omega_W)
- a_{\mathcal{U}_\varepsilon}(\gamma_Z,\omega_W) \;,\\
a_{U_\varepsilon'}(f^*\gamma,\omega_{X'})
- a_{\mathcal{U}_\varepsilon'}(g^*\gamma_Z,\omega_{W'}) & =
a_{\mathcal{U}_\varepsilon'}(g^*\varphi_*\gamma,\omega_{W'})
- a_{\mathcal{U}_\varepsilon'}(g^*\gamma_Z,\omega_{W'}) \;.
\end{split}
\end{align}

By the definition of $\gamma_Z$ (see \eqref{eq12-pf-intro-thm-bl-curve})
and \eqref{eq2t-pf-intro-thm-bl-curve},
there exists $\alpha\in H^0(\mathcal{U}_\varepsilon,\C)$ such that
\begin{equation}
\label{eq43-pf-intro-thm-bl-curve}
\varphi_*\gamma = e^\alpha \gamma_Z \;,\hspace{5mm}
\alpha\big|_Z = 0 \;.
\end{equation}
Let $g^{TW}$ (resp. $g^{TW'}$) be the metric on $TW$ (resp. $TW'$)
induced by $\omega_W$ (resp. $\omega_{W'}$).
By \eqref{eq-def-aX} and \eqref{eq41-pf-intro-thm-bl-curve}-\eqref{eq43-pf-intro-thm-bl-curve},
we have
\begin{align}
\label{eq44-pf-intro-thm-bl-curve}
\begin{split}
& a_{X'}(f^*\gamma,\omega_{X'}) - a_X(\gamma,\omega_X)
- a_{W'}(g^*\gamma_Z,\omega_{W'}) + a_W(\gamma_Z,\omega_W) \\
& = \Big( a_{\mathcal{U}_\varepsilon'}(g^*\varphi_*\gamma,\omega_{W'})
- a_{\mathcal{U}_\varepsilon'}(g^*\gamma_Z,\omega_{W'}) \Big)
- \Big( a_{\mathcal{U}_\varepsilon}(\varphi_*\gamma,\omega_W)
- a_{\mathcal{U}_\varepsilon}(\gamma_Z,\omega_W) \Big) \\
& = \Big( a_{\mathcal{U}_\varepsilon'}(g^*(e^\alpha\gamma_Z),\omega_{W'})
- a_{\mathcal{U}_\varepsilon'}(g^*\gamma_Z,\omega_{W'}) \Big)
- \Big( a_{\mathcal{U}_\varepsilon}(e^\alpha\gamma_Z,\omega_W)
- a_{\mathcal{U}_\varepsilon}(\gamma_Z,\omega_W) \Big) \\
& = \frac{1}{12} \int_{\mathcal{U}_\varepsilon'} c_3\big(TW',g^{TW'}\big) \log \bigg|\frac{g^*(e^\alpha\gamma_Z)}{g^*\gamma_Z}\bigg|^2
- \frac{1}{12} \int_{\mathcal{U}_\varepsilon} c_3\big(TW,g^{TW}\big) \log \bigg|\frac{e^\alpha\gamma_Z}{\gamma_Z}\bigg|^2 \\
& = \frac{1}{12} \int_{\mathcal{U}_\varepsilon'} c_3\big(TW',g^{TW'}\big) g^*(\alpha+\overline{\alpha})
- \frac{1}{12} \int_{\mathcal{U}_\varepsilon} c_3\big(TW,g^{TW}\big) (\alpha+\overline{\alpha}) \\
& = \frac{1}{12} \int_{\mathcal{U}_\varepsilon'}
\Big( c_3\big(TW',g^{TW'}\big) - g^* c_3\big(TW,g^{TW}\big) \Big) g^* (\alpha+\overline{\alpha}) \;.
\end{split}
\end{align}
Proceeding in the same way as in the proof of Theorem \ref{thm-ind-omega},
we see that the right hand side of \eqref{eq44-pf-intro-thm-bl-curve}
is independent of $g^{TW}$ and $g^{TW'}$
as long as they coincide on the boundary of $\mathcal{U}_\varepsilon$.

Recall that $\mathcal{U}_\varepsilon$ was defined by \eqref{eq2ue-pf-intro-thm-bl-curve}.
For $0<\epsilon<1$,
we define a biholomorphic map
\begin{align}
\begin{split}
r_\epsilon: \mathcal{U}_\varepsilon & \rightarrow \mathcal{U}_{\epsilon\varepsilon} \\
v & \mapsto \epsilon v \;.
\end{split}
\end{align}
Let
\begin{equation}
r'_\epsilon: \mathcal{U}_\varepsilon' \rightarrow \mathcal{U}_{\epsilon\varepsilon}'
\end{equation}
be the lift of $r_\epsilon$.
Let $g^{TW}_\epsilon$ and $g^{TW'}_\epsilon$ be metrics on $TW$ and $TW'$
such that they coincide on $\mathcal{U}_\varepsilon\backslash\mathcal{U}_{\epsilon\varepsilon}$.
Since the right hand side of \eqref{eq44-pf-intro-thm-bl-curve} is independent of $g^{TW}$ and $g^{TW'}$,
we have
\begin{align}
\label{eq47-pf-intro-thm-bl-curve}
\begin{split}
& \int_{\mathcal{U}_\varepsilon'}
\Big( c_3\big(TW',g^{TW'}\big) - g^* c_3\big(TW,g^{TW}\big) \Big)
g^* (\alpha+\overline{\alpha})  \\
& = \int_{\mathcal{U}_\varepsilon'}
\Big( c_3\big(TW',g^{TW'}_\epsilon\big) - g^* c_3\big(TW,g^{TW}_\epsilon\big) \Big)
g^* (\alpha+\overline{\alpha}) \\
& = \int_{\mathcal{U}_{\epsilon\varepsilon}'}
\Big( c_3\big(TW',g^{TW'}_\epsilon\big) - g^* c_3\big(TW,g^{TW}_\epsilon\big) \Big)
g^* (\alpha+\overline{\alpha})\\
& = \int_{\mathcal{U}_\varepsilon'}
\Big( c_3\big(TW',(r_\epsilon')^*g^{TW'}_\epsilon\big)
- g^* c_3\big(TW,r_\epsilon^*g^{TW}_\epsilon\big) \Big)
g^* r_\epsilon^* (\alpha+\overline{\alpha}) \\
& = \int_{\mathcal{U}_\varepsilon'}
\Big( c_3\big(TW',g^{TW'}\big) - g^* c_3\big(TW,g^{TW}\big) \Big)
g^* r_\epsilon^* (\alpha+\overline{\alpha}) \;.
\end{split}
\end{align}
Taking $\epsilon\rightarrow 0$ in \eqref{eq47-pf-intro-thm-bl-curve}
and applying the second identity in \eqref{eq43-pf-intro-thm-bl-curve},
we get
\begin{equation}
\label{eq48-pf-intro-thm-bl-curve}
\int_{\mathcal{U}_\varepsilon'}
\Big( c_3\big(TW',g^{TW'}\big) - g^* c_3\big(TW,g^{TW}\big) \Big)
g^* (\alpha+\overline{\alpha}) = 0 \;.
\end{equation}
From \eqref{eq44-pf-intro-thm-bl-curve} and \eqref{eq48-pf-intro-thm-bl-curve},
we obtain \eqref{eq4-pf-intro-thm-bl-curve}.

\noindent\textbf{Step 5.}
We conclude.

By \eqref{eq-def-tau-omega-aux}, \eqref{eq-def-tau-aux}
and \eqref{eq22-pf-intro-thm-bl-curve},
we have
\begin{align}
\label{eq51-pf-intro-thm-bl-curve}
\begin{split}
& \tau(X',f^*\gamma) - \tau(X,\gamma) \\
& = \tau_\mathrm{BCOV}(X',\omega_{X'}) - \tau_\mathrm{BCOV}(X,\omega_X) - \frac{1}{2} \tau_\mathrm{BCOV}(D,\omega_{X'}\big|_D) \\
& \hspace{5mm} + \frac{1}{2} a_D(f^*\gamma,\omega_{X'}) + b_D(\omega_{X'})
- a_{X'}(f^*\gamma,\omega_{X'}) + a_X(\gamma,\omega_X) \;.
\end{split}
\end{align}
By \eqref{eq-def-tau-omega-aux}, \eqref{eq-def-tau-aux}
and the first identity in \eqref{eq23-pf-intro-thm-bl-curve},
we have
\begin{align}
\label{eq52-pf-intro-thm-bl-curve}
\begin{split}
& \tau(W',g^*\gamma_Z) - \tau(W,\gamma_Z) \\
& = \tau_\mathrm{BCOV}(W',\omega_{W'}) - \tau_\mathrm{BCOV}(W,\omega_W) - \frac{1}{2} \tau_\mathrm{BCOV}(D,\omega_{W'}\big|_D) \\
& \hspace{5mm} + \frac{1}{2} a_D(g^*\gamma_Z,\omega_{W'}) + b_D(\omega_{W'})
- a_{W'}(g^*\gamma_Z,\omega_{W'}) + a_W(\gamma_Z,\omega_W) \;.
\end{split}
\end{align}
By the second identity in \eqref{eq23-pf-intro-thm-bl-curve},
we have
\begin{equation}
\label{eq53-pf-intro-thm-bl-curve}
\tau_\mathrm{BCOV}(D,\omega_{X'}\big|_D) = \tau_\mathrm{BCOV}(D,\omega_{W'}\big|_D) \;,\hspace{5mm}
b_D(\omega_{X'}) = b_D(\omega_{W'}) \;.
\end{equation}
By the definition of $\gamma_Z$ (see \eqref{eq12-pf-intro-thm-bl-curve})
and the second identity in \eqref{eq23-pf-intro-thm-bl-curve},
we have
\begin{equation}
\label{eq54-pf-intro-thm-bl-curve}
a_D(f^*\gamma,\omega_{X'}) =
a_D(g^*\gamma_Z,\omega_{W'}) \;.
\end{equation}
From \eqref{eq3-pf-intro-thm-bl-curve}, \eqref{eq4-pf-intro-thm-bl-curve}
and \eqref{eq51-pf-intro-thm-bl-curve}-\eqref{eq54-pf-intro-thm-bl-curve},
we obtain \eqref{eq1-pf-intro-thm-bl-curve}.
This completes the proof.
\end{proof}

\providecommand{\bysame}{\leavevmode\hbox to3em{\hrulefill}\thinspace}
\providecommand{\MR}{\relax\ifhmode\unskip\space\fi MR }
\providecommand{\MRhref}[2]{%
  \href{http://www.ams.org/mathscinet-getitem?mr=#1}{#2}
}
\providecommand{\href}[2]{#2}

\end{document}